\documentclass[a4paper,12pt]{article}
\usepackage[utf8]{inputenc}
\usepackage{latexsym,amsfonts,amsmath,amssymb,mathrsfs,url,amsthm}

\usepackage{mathtools}
\usepackage{color,xcolor,graphicx}
\usepackage{mathdots}
\usepackage{comment}
\usepackage{hyperref}
\usepackage{enumitem}

\theoremstyle{plain} 
\newtheorem{theorem}{Theorem}[section]
\newtheorem{lemma}[theorem]{Lemma}

\newtheorem{corollary}[theorem]{Corollary}

\theoremstyle{definition} 
\newtheorem{definition}[theorem]{Definition}
\newtheorem{remark}[theorem]{Remark}

\providecommand{\keywords}[1]
{
  \noindent \small	
  \textbf{Keywords:} #1
}
\providecommand{\amscode}[1]
{
  \noindent \small	
  \textbf{AMS subject classifications:} #1
}
\newcommand{\bsa}{\boldsymbol{a}}

\newcommand{\bsg}{\boldsymbol{g}}
\newcommand{\bsh}{\boldsymbol{h}}
\newcommand{\bsk}{\boldsymbol{k}}
\newcommand{\bst}{\boldsymbol{t}}
\newcommand{\bsx}{\boldsymbol{x}}
\newcommand{\bsy}{\boldsymbol{y}}
\newcommand{\bsz}{\boldsymbol{z}}
\newcommand{\bszero}{\boldsymbol{0}}

\newcommand{\bsdelta}{\boldsymbol{\delta}}

\newcommand{\NN}{\mathbb{N}}
\newcommand{\RR}{\mathbb{R}}
\newcommand{\XX}{\mathbb{X}}
\newcommand{\ZZ}{\mathbb{Z}}

\newcommand{\Scal}{\mathcal{S}}

\allowdisplaybreaks

\renewcommand{\det}{{\rm det}\,} 

\newcommand{\uu}{\mathfrak{u}}

\providecommand{\keywords}[1]
{
  \noindent \small	
  \textbf{Keywords:} #1
}
\providecommand{\amscode}[1]
{
  \noindent \small	
  \textbf{AMS subject classifications:} #1
}

\title{On the quasi-uniformity properties of quasi-Monte Carlo point sets and sequences -- Part~I: Lattices and Kronecker sequences
\thanks{This work of J.D. is supported by ARC grant DP220101811. The work of T.G. is supported by JSPS KAKENHI Grant Number 23K03210. The work of K.~S. is supported by JSPS KAKENHI Grant Number 24K06857.}}
\author{Josef Dick\thanks{School of Mathematics and Statistics, The University of New South Wales Sydney, 2052 NSW, Australia (\url{josef.dick@unsw.edu.au})},
Takashi Goda\thanks{Graduate School of Engineering, The University of Tokyo, 7-3-1 Hongo, Bunkyo-ku, Tokyo 113-8656, Japan (\url{goda@frcer.t.u-tokyo.ac.jp})}, 
Gerhard Larcher\thanks{Institute for Financial Mathematics and Applied Number Theory, Johannes Kepler University Linz, Altenbergerstra{\ss}e 69, 4040 Linz, Austria (\url{gerhard.larcher@jku.at})},\\ 
Friedrich Pillichshammer\thanks{Institute for Financial Mathematics and Applied Number Theory, Johannes Kepler University Linz, Altenbergerstra{\ss}e 69, 4040 Linz, Austria (\url{friedrich.pillichshammer@jku.at})}, 
Kosuke Suzuki\thanks{Faculty of Science, Yamagata University, 1-4-12 Kojirakawa-machi, Yamagata, 990-8560, Japan (\url{kosuke-suzuki@sci.kj.yamagata-u.ac.jp})}}
\date{\today}

\begin{document}

\maketitle

\amscode{11K36, 11K38, 11J71}

\begin{abstract}
The discrepancy of a point set quantifies how well the points are distributed, with low-discrepancy point sets demonstrating exceptional uniform distribution properties. Such sets are integral to quasi-Monte Carlo methods, which approximate integrals over the unit cube for integrands of bounded variation. In contrast, quasi-uniform point sets are characterized by optimal separation and covering radii, making them well-suited for applications such as radial basis function approximation. This paper explores the quasi-uniformity properties of quasi-Monte Carlo point sets constructed from lattices and also Kronecker sequences. Specifically, we analyze rank-1 lattice point sets, Fibonacci lattice point sets, Frolov point sets, and Kronecker sequences (also referred to as $(n \boldsymbol{\alpha})$-sequences), providing insights into their potential for use in applications that require both low-discrepancy and quasi-uniform distribution. As an example, we show that the $(n \boldsymbol{\alpha})$-sequence with $\alpha_j = 2^{j/(d+1)}$ for $j \in \{1, 2, \ldots, d\}$ is quasi-uniform and has low-discrepancy. The quasi-uniformity properties of quasi-Monte Carlo digital nets and sequences will be studied in a companion paper.
\end{abstract}


\sloppy

\section{Introduction}
In this paper, together with the companion paper \cite{DGSxx}, we study two properties of point sets and sequences that frequently arise in numerical analysis:
quasi-uniformity on the one hand, and uniform distribution modulo~$1$ (for sequences) as well as small star-discrepancy (for finite point sets) on the other hand.
Our main goal is to understand to what extent these two requirements can be achieved simultaneously within structured QMC constructions.

Quasi-uniformity is a geometric space-filling condition.
It requires that the largest ``holes'' are small (small covering radius) while points do not cluster excessively (large separation radius).
For $N$ points in dimension $d$, the optimal regime is that both radii scale like $N^{-1/d}$, and a convenient quantitative proxy is the mesh ratio, i.e., the quotient of covering radius and separation radius.
In contrast, uniform distribution modulo~$1$ and discrepancy quantify distributional uniformity and, in general, need not control geometric separation; see Section~\ref{ssec:quvsud}.

This issue is particularly relevant when the same nodes are used both for QMC-type sampling and for geometric approximation tasks; we discuss this motivation in Section~\ref{ssec:motiv}.
From the viewpoint of QMC constructions, however, the relationship between low-discrepancy (or uniform distribution modulo~$1$) and quasi-uniformity has so far not been studied in a unified manner, and the available results are scattered across different settings.
Accordingly, we investigate to what extent quasi-uniformity (in the form of bounded mesh ratio) can be combined with uniform distribution and small star-discrepancy within structured QMC constructions.

The present paper (Part~I) focuses on lattice-type constructions.
Our main contribution is to provide usable algebraic conditions ensuring bounded mesh ratio for several lattice-based node families, and to combine these conditions with discrepancy bounds in order to obtain explicit examples where bounded mesh ratio and small star-discrepancy hold simultaneously.
This covers, in particular, suitably scaled and shifted admissible (Frolov-type) lattices, classical two-dimensional lattice point sets, and Kronecker sequences, for which we also characterize the quasi-uniformity behavior.
Moreover, for rank-1 lattice point sets in arbitrary dimension we prove an existence result, showing that there exist generating vectors for which both bounded mesh ratio and small star-discrepancy are attained.

Related questions for digital nets and digital sequences are addressed separately in a companion work (Part~II), see \cite{DGSxx}.

We now give the precise definitions of the two concepts under consideration.

\subsection{Quasi-Uniformity}

Let $\Omega \subseteq \mathbb{R}^d$ be a compact subset with $\mathrm{vol}(\Omega) > 0$. Let $P$ be a point set in $\Omega$. For $p \in [1,\infty]$ let $\|\cdot \|_p$ denote the $\ell_p$ norm, i.e., $\|\bsx\|_p=(|x_1|^p+\cdots+|x_d|^p)^{1/p}$ for $\bsx=(x_1,\ldots,x_d) \in \mathbb{R}^d$, with the obvious modifications for $p = \infty$.

Define the {\it covering radius} (in $\ell_p$ norm) by
\[ h_p(P; \Omega):=\sup_{\bsx\in \Omega}\min_{\bsy\in P} \|\bsx-\bsy\|_p. \]
The covering radius can be described in the following way: If one places closed $\ell_p$ balls of radius $r$ around each point of $P$, then the smallest value of $r$ such that the union of the balls cover $\Omega$ is called the covering radius. In the literature, the covering radius is also known as {\it dispersion} (see, e.g., \cite[Definition~6.2]{N92}), as {\it fill distance} (see, e.g., \cite{H21}), as {\it mesh norm}, or as {\it minimax-distance criterion} (see \cite{PZ23}).

The {\it separation radius} (in $\ell_p$ norm) is given by
\[ q_p(P; \Omega):=\min_{\substack{\bsx,\bsy\in P\\ \bsx\neq \bsy}} \frac{\|\bsx-\bsy\|_p}{2}.\]
The separation radius can be described in the following way: If one places open $\ell_p$ balls of radius $r$ around each point of $P$, then the separation radius is the largest value of $r$ such that none of these balls intersect in $\Omega$. In the literature, the separation radius is also known as {\it packing radius} or {\it maximin-distance criterion}; see \cite{PZ23}.

The {\it mesh ratio} (in $\ell_p$ norm) is given by
\begin{equation*}
\rho_p(P; \Omega) = \frac{h_p(P; \Omega)}{q_p(P; \Omega)}.
\end{equation*}

Since we are mainly dealing with the domain $[0,1]^d$ we write $h_p(P), q_p(P), \rho_p(P)$ instead of $h_p(P; [0,1]^d), q_p(P; [0,1]^d), \rho_p(P; [0,1]^d)$ for brevity.

It is obvious that $\rho_p(P;\Omega)\ge 1$, at least if $\Omega$ is connected. Furthermore, it is known that the optimal order of the separation radius and the covering radius is $N^{-1/d}$. In more detail, we have $$h_p(P,\Omega) \ge \frac{c_{p,d}}{N^{1/d}} \quad \mbox{ and } \quad q_p(P,\Omega) \le \frac{c_{p,d}'}{N^{1/d}},$$ for quantities $c_{p,d},c_{p,d}'>0$, where the lower bound on the separation radius is taken from \cite[Lemma~2.1]{PZ23} and the upper bound on the covering radius follows by comparing the volume of the balls around each point in $P$ to the volume of $\Omega$.
An explicit bound for $\Omega = [0,1]^d$ is given in Appendix~\ref{sec:unitcube_bound}.
Hence the mesh ratio of a family of point sets with an increasing number of points is bounded by a constant independent of the number of points if and only if both the covering radii and the separation radii are of the optimal order. 
This motivates the notion of quasi-uniformity, i.e., bounded mesh ratio.
In the literature, quasi-uniformity is defined in two closely related ways.
Some works formulate it for a \emph{set/sequence of data sets} $(P_i)_{i\ge1}$ with increasing resolution, which does not need to be extensible.
Other works define it for an infinite \emph{sequence of points} $(x_n)_{n\ge1}$ by imposing the same property on the initial segments $P_N=\{x_1,\dots,x_N\}$.

In this paper, we state both formulations explicitly and further define the notion of ``quasi-uniform along a subsequence'' as below.
This notion is motivated by the way QMC point sets are used in practice.
Some standard QMC sequences, such as Sobol' sequences, are naturally organized in dyadic sample sizes $N=2^m$, and one often evaluates or refines approximations only at these values of $N$.
More generally, even if a full sequence does not exhibit uniformly good geometric regularity for every $N$, it is natural to focus on those subsequence lengths for which the mesh ratio is well controlled.

\begin{definition}\label{def_qu_pointsets}
    Let $X$ be a set of point sets in $\Omega$ with $\sup_{P \in X} |P| = \infty$.
    Then we call $X$ \emph{a quasi-uniform family} (of point sets) if there exists a constant $C_{p} >0$ such that the mesh ratio is bounded, i.e.
    \[ \rho_p(P) \leq C_{p}, \quad \forall P  \in X.\]
\end{definition}

\begin{definition}
Let $\Scal = (\bsx_n)_{n \ge 0}$ be an infinite sequence of points. We denote the set of the first $i$ points of $\Scal$ by $P_i =\{\bsx_0, \bsx_1, \ldots, \bsx_{i-1}\}$. Then, we call $\Scal$ a \emph{quasi-uniform sequence} if the set $\{P_i\}_{i \in \NN}$ is a quasi-uniform family, i.e., there exists a constant $C'_p > 0$ such that \[ \rho_p(P_i) \le C'_p, \quad \forall i \in \NN\setminus \{1\}.\]
We call the sequence $\Scal$ \emph{quasi-uniform along the subsequence} $1 \le i_1 < i_2 < \cdots$ if the mesh ratio is bounded for all $i_k$, i.e., $\rho_p(P_{i_k}) \le C'_p$ fo all $k \in \NN$.
\end{definition}

In fact, the constant can be chosen independently of the dimension~$d$. It is shown in \cite{PZ23} that there exists a quasi-uniform sequence with $C'_p = 2$, and that this constant is the best possible; no sequence can satisfy $\rho_p(P_i)<2$ for all $i$.
In contrast, in our framework the constants typically depend on~$d$, and the resulting bounds are mainly effective in moderate dimensions. We leave a systematic investigation of the dimension dependence of the mesh ratio for QMC point sets to future work.

Since $\bsx, \bsy \in [0,1]^d$ are finite dimensional vectors, all norms $\| \cdot \|_p$ are equivalent, therefore, if a family of point sets is quasi-uniform for one $p$, it is quasi-uniform for all $p \in [1,\infty]$.

In the following lemma we show that quasi-uniformity along a subsequence whose indices grow at most geometrically upgrades to quasi-uniformity of the full sequence, and vice versa.
The proofs are given in Appendix~\ref{sec:app_proof1} and \ref{sec:app_proof2}.

\begin{lemma}\label{lem_seq_weak}
Let $\Scal = \{\bsx_0, \bsx_1, \ldots \} \subseteq [0,1]^d$ be a sequence. Let $P_i = \{ \bsx_0, \bsx_1, \ldots, \bsx_{i-1}\}$ denote the first $i$ points of $\Scal$. Let $1 \le i_1 < i_2 < i_3 < \cdots$ be an increasing sequence at most geometrically, i.e., there is a constant $c > 1$ such that $i_{k+1} \le c i_{k}$ for all $k \in \mathbb{N}$.

\begin{enumerate}
\item 
Assume that $\Scal$ is a quasi-uniform sequence along the subsequence $i_1 < i_2 < \cdots$, i.e., the mesh ratio $\rho_p(P_{i_k}) \le C''_p$  for all $k \in \mathbb{N}$.
Then $\Scal$ is a quasi-uniform sequence.

\item Assume that the covering radius $h_p(P_{i_k}) \le C_p i_k^{-1/d}$ for all $k \in \mathbb{N}$ for some constant $C_p > 0$. 
Then there is a constant $D_p$ such that $h_p(P_i) \le D_p i^{-1/d}$ for all $i \in \mathbb{N}$.

\item Assume that the separation radius $q_p(P_{i_k}) \ge C'_p i_k^{-1/d}$ for all $k \in \mathbb{N}$ for some constant $C'_p > 0$.
Then there is a constant $D'_p$ such that $q_p(P_i) \ge D'_p i^{-1/d}$ for all $i \in \mathbb{N}$.
\end{enumerate}
\end{lemma}

\begin{lemma}\label{lem_seq_weak_inv}
Let $1 \le i_1 < i_2 < i_3 < \cdots$ be an increasing sequence with $\sup_k i_{k+1}/i_k = \infty$.
Then there exists a sequence of points $\Scal$ which is quasi-uniform along the subsequence $1 \le i_1 < i_2 < \cdots$
but not a quasi-uniform sequence.
\end{lemma}

\medskip \textbf{Brief summary of known results.} Quasi-uniform point sets play an important role for instance in the areas of design of computer experiments \cite{FLS06, PM12, SWN03} and radial basis function approximation \cite{Sch95,W05}.

The existence of quasi-uniform infinite sequences was shown in \cite{PZ23}, which proves that such sequences can be constructed by a greedy packing algorithm. It is also shown that there are sequences such that the mesh ratio of the first $N$ points of the sequence is at most $2$ for all $N \ge 2$.
The resulting sequences do not necessarily have low discrepancy.
The case of $d=1$ is an exception, where the resulting sequences, depending on the initialization of the greedy packing, coincide with the van der Corput sequence in base $2$, along with the additional point at $x=1$.
A numerical illustration will be given in Part II \cite{DGSxx}.
Another concrete explicit example in dimension~1 is the Kronecker sequence $(\{n \alpha\})_{n \ge 0}$, which is quasi-uniform if and only if $\alpha$ is a badly approximable number; see \cite{G24b} and Section~\ref{secc:nalpha} of the present paper. 

On the other hand, important sequences that seem at first sight to be candidates for quasi-uniform sequences do actually not possess this property. An example for such an instance is the Sobol' sequence in dimension 2.
In \cite{SS07}, it was predicted based on numerical experiments that the Sobol' sequence satisfies a lower bound on the separation radius of order $N^{-1/d}$ for any dimension $d$.
This was shown not to hold for dimension $d=2$ in \cite{G24a},
and thus the two-dimensional Sobol' sequence is not quasi-uniform.
More strongly, it is not quasi-uniform along any subsequence \cite{suzuki2025a}.
However, the case for $d \ge 3$ remains open.

\subsection{Uniform distribution modulo 1 and discrepancy}\label{subsec:ud1}

In 1916 Weyl~\cite{Weyl16} introduced the concept of uniform distribution modulo 1 of sequences. Accordingly, an infinite sequence $(\bsx_n)_{n \ge 0}$ in $\mathbb{R}^d$ is called {\it uniformly distributed modulo 1}, if for every interval $[\boldsymbol{a},\boldsymbol{b}) \subseteq [0,1)^d$ we have $$\lim_{N \rightarrow \infty} \frac{|\{n < N\ : \ \{\bsx_n\} \in [\boldsymbol{a},\boldsymbol{b})\}|}{N}={\rm Volume}([\boldsymbol{a},\boldsymbol{b})),$$ where $\{x\}=x-\lfloor x \rfloor$ is $x$ taken modulo $1$, which is applied component-wise if the argument is a vector, i.e., $\{\bsx\}=(\{x_1\},\ldots,\{x_d\})$ for $\bsx=(x_1,\ldots,x_d) \in \mathbb{R}^d$. Since we take all numbers modulo $1$ all points can be understood to lie in $[0,1)^d$. Hence we consider point sets in $[0,1)^d$ in the following directly, thus we leave out modulo $1$ and simply refer to uniform distribution. For $\bsx \in [0,1)^d$ we also do not need $\{ \cdot \}$. For more information on uniform distribution, which builds the fundamental basis of quasi-Monte Carlo theory, we refer to \cite{DP10,DT97,KN74,LePi14}.

A quantitative measure for how close the empirical distribution of a given point set in $[0,1)^d$ is to the uniform distribution is the notion of discrepancy. For a class $\mathcal{C}$ of measurable test sets in $[0,1)^d$ the discrepancy (with respect to $\mathcal{C}$) of an $N$-point set $P_N=\{\bsx_0,\bsx_1,\ldots,\bsx_{N-1}\}$ in $[0,1)^d$ is defined as $$D^{\mathcal{C}}_N(P_N)=\sup_{C \in \mathcal{C}} \left|\frac{|\{0 \le n < N \ : \ \bsx_n \in C\}|}{N}-{\rm Volume}(C) \right|.$$ For an infinite sequence $\Scal=(\bsx_n)_{n \ge 1}$ the discrepancy $D_N^{\mathcal{C}}(\Scal)$ is defined as the discrepancy $D_N^{\mathcal{C}}(P_N)$ of the initial segment $P_N=\{\bsx_0,\bsx_1,\ldots,\bsx_{N-1}\}$ of $\Scal$. 

Popular notions of discrepancy are the so-called {\it star-discrepancy} $D_N^*$, which refers to the class of all intervals of the form $[\bszero,\bsa)$ with $\bsa \in [0,1]^d$ as test sets, or the {\it isotropic discrepancy} $J_N$, where the class of test sets is given by all convex subsets $C$ of $[0,1)^d$. Those two notions are related in the following way: For every dimension $d \in \NN$ there exists a quantity $C_d>0$ such that for every $N$-point set $P_N$ in $[0,1)^d$ we have $$D_N^*(P_N) \le J_N(P_N) \le C_d (D_N^{\ast}(P_N))^{1/d},$$ see, e.g., \cite[Chapter~2, Section~1, Theorem~1.6]{KN74}. Furthermore, it is well-known that a sequence $\Scal$ in $[0,1)^d$ is uniformly distributed if and only if its star-discrepancy $D_N^{\ast}(\Scal)$ tends to 0 for $N$ tending to infinity; see, e.g., \cite[Theorem~2.15]{LePi14}. The same holds true with the star-discrepancy replaced by isotropic discrepancy.

With some abuse of nomenclature, we call a sequence $(P_{N_i})_{i \ge 1}$ of $N_i$-element point sets $P_{N_i}$ in $[0,1)^d$ with $N_1 < N_2 < N_3 < \dots$ uniformly distributed modulo 1, if $$\lim_{i \rightarrow \infty} D_{N_i}^*(P_{N_i})=0.$$

\medskip \textbf{Brief summary of known results.} The following asymptotic results about discrepancies are known: 
\begin{itemize}[leftmargin = *]
\item Star-discrepancy: For every $N$-point set $P_N$ in $[0,1)$ we have $D_N^\ast (P_N) \ge 1/(2N)$. For every $d\in \NN\setminus\{1\}$ there exist quantities $c_d>0$ and $\delta_d>0$ such that for every $N$-point set $P_N$ in $[0,1)^d$ we have $$D_N^{\ast}(P_N) \ge c_d \frac{(\log N)^{\frac{d-1}{2}+\delta_d}}{N},$$ where $\delta_2=1/2$ and $\lim_{d \rightarrow \infty} \delta_d=0$. This result, which is an important improvement of Roth's lower bound \cite{roth1}, was shown in \cite{blv08}. The result for $d=2$ is according to \cite{schm72}. On the other hand, for every $N \in \NN \setminus\{1\}$ there exists an $N$-point set $P_N$ in $[0,1)^d$, such that $$D_N^{\ast}(P_N) \ll_d \frac{(\log N)^{d-1}}{N}.$$ Examples are constructions according to Hammersley, Faure, Niederreiter, and others. See, e.g., \cite{DP10} for an overview. It is still an open question whether the order of magnitude $(\log N)^{d-1}/N$ is best possible for the star-discrepancy of $N$-element point sets in $[0,1)^d$.

It has become common to speak about ``low-discrepancy'' sequences or point sets if the corresponding star-discrepancy is of order of magnitude $(\log N)^{\kappa(d)}/N$ for a ``small'' (linearly in $d$) exponent $\kappa(d)$ or, slightly weaker, of order of magnitude $N^{-1+\delta}$ for every $\delta>0$. Low-discrepancy sequences and point sets are of utmost importance as underlying nodes of quasi-Monte Carlo integration rules; see, e.g., \cite{DP10,DT97,KN74,LePi14,N92}.

\item Isotropic discrepancy: In \cite[Theorem~1]{schm75} (see also \cite[Theorem~13A]{schm77}) Schmidt proved the following general lower bound for arbitrary point sets: For every dimension $d$ there exists a quantity $c_d>0$ such that for all $N$-element point sets $P_N$ in dimension $d$ we have $$J_N(P_N) \ge \frac{c_d}{N^{2/(d+1)}}.$$ This result is essentially (up to $\log$-factors) best possible as shown in \cite{beck} for $d=2$ and in \cite{stute} for $d\ge 3$ using probabilistic methods.

For $N$-element lattice point sets $P_N$ in $[0,1)^d$ (which are the main focus of this work -- for a formal definition see Section~\ref{sec:lat}) we always have $$J_N(P_N) \ge \frac{c_d}{N^{1/d}},$$ for some $c_d>0$ depending only on $d$, as shown in \cite{PS20}.

Upper bounds on the isotropic discrepancy of special point sets are given in \cite{lar86,L88}. For example, $J_N(P_N) \ll_d N^{-1/d}$ if $P_N$ is the $N$-element Hammersley net or the $N$-th initial segment of the Halton sequence in dimension $d$.
\end{itemize}

\subsection{Quasi-uniformity versus uniform distribution}\label{ssec:quvsud}

We now present examples that demonstrate the variety of scenarios that can arise concerning the quasi-uniformity and discrepancy of a point set. Here and throughout the paper, we use the following notation: for functions $f,g:\NN \rightarrow \RR$ we write $f(N) \ll g(N)$ if there exists a quantity $C > 0$ independent of $N$ such that $f(N) \le C \ g(N)$ for all $N \in \NN$. We stress that $f(N) \ll g(N)$ has the same meaning as the for some more familiar notation $f(N)=O(g(N))$. If we want to stress that $C$ depends on some parameters, say $d,p$, then this is indicated by writing $f(N) \ll_{d,p} g(N)$. If we have simultaneously $f(N) \ll g(N)$ and $g(N) \ll f(N)$ (of course with different involved quantities $C$), then we write $f(N) \asymp g(N)$ or, if a dependence on other parameters, say $d,p$, should be stressed, $f(N) \asymp_{d,p} g(N)$. We also stress that $f(N) \asymp g(N)$ has the same meaning as $f(N)=\Theta(g(N))$.

\medskip \textbf{Quasi-uniformity does not imply uniform distribution.} Although the name quasi-uniform seems to suggest that quasi-uniformity is related to uniform distribution, it is well understood that this is actually not the case in general. For example, already in dimension $d=1$ the point sets 
\begin{equation*}
\left\{0, \frac{1}{2N}, \ldots, \frac{N-1}{2N}\right\} \cup \left\{\frac{1}{2}, \frac{1}{2}+\frac{1}{4N}, \ldots, \frac{1}{2}+\frac{2N}{4N}= 1\right\} 
\end{equation*}
for $N \in \NN$, are a quasi-uniform family, but are not uniformly distributed, since the star-discrepancy of this point set does not converge to $0$ as $N \to \infty$.

\medskip \textbf{A sequence of uniformly distributed lattice point sets which are not quasi-uniform.} Let $d, m \in \mathbb{N}$, $d > 1$ and let $n_1 = \cdots = n_{d-1} = m$ and $n_d = m^2$. Let $\boldsymbol{n} = \boldsymbol{n}(m) = (n_1, n_2, \ldots, n_d) $ and
\begin{equation*}
\Gamma_{\boldsymbol{n}} = \left\{ \left(\frac{k_1}{n_1}, \ldots, \frac{k_d}{n_d} \right): 0 \le k_j < n_j,\ j= 1, 2, \ldots, d  \right\},
\end{equation*}
be the uniform grid with mesh sizes $n_1,\ldots,n_d$. The total number of points $N = |\Gamma_{\boldsymbol{n}} |= \prod_{j=1}^d n_j= m^{d+1}$. Consider a box $B_{\bst} = \prod_{j=1}^d [0, t_j)$ for some $\bst = (t_1, \ldots, t_d) \in (0,1]^d$. Let $k_j \in \mathbb{N}$ be the unique integer such that $k_j/n_j < t_j \le (k_j+1)/n_j$, i.e., $k_j= \lceil t_j n_j \rceil - 1$. Then the number of points from $\Gamma_{\boldsymbol{n}}$ that belong to $B_{\bst}$ is $k_1 k_2 \cdots k_d$ and
\begin{equation*}
\lim_{m\to \infty} \frac{|B_{\bst} \cap \Gamma_{\boldsymbol{n}}|}{N} = \lim_{m\to \infty} \prod_{j=1}^d \frac{\lceil t_j n_j \rceil -1}{n_j} = \prod_{j=1}^d t_j = \mathrm{Volume}(B_{\bst}).
\end{equation*}
Hence, the sequence of point sets  $(\Gamma_{\boldsymbol{n}(m)})_{m \ge 1}$ is uniformly distributed.

On the other hand, the separation radius is 
\begin{equation*}
q_{p}(\Gamma_{\boldsymbol{n}(m)}) = \frac{1}{2 n_d} = \frac{1}{2 m^2} = \frac{1}{2 N^{\frac{2}{d+1}}} = \frac{1}{2 N^{\frac{1}{d}+\frac{d-1}{d(d+1)}}}.
\end{equation*}
Since the separation is not of the optimal order (which would be $N^{-1/d}$), it follows that the family $(\Gamma_{\boldsymbol{n}(m)})_{m \ge 1}$ is not quasi-uniform.

\medskip \textbf{A quasi-uniform family of lattice point sets is always uniformly distributed.} 
Lattice point sets are the major object of interest in this paper. The formal definition will be presented in Section 2. For the time being the following suffices: For $i \in \mathbb{N}$ let $T_i \in \mathbb{R}^{d \times d}$ and
\begin{equation*}
\XX_i = T_i(\ZZ^d)=\{ T_i \bsk \ : \ \bsk\in \ZZ^d \}
\end{equation*}
and $P_{N_i} = \XX_i \cap [0,1)^d$ with $N_i = |P_i|$ and $N_1 < N_2 < \dots$. Assume that the mesh ratio $\rho_{p}(P_{N_i}) \le C < \infty$ for some constant $C$ independent of $N_i$. Then, necessarily, $h_p(P_{N_i}) \asymp N_i^{-1/d}$.
From Lemma~\ref{le:equiv} (see Section~\ref{sec:lat} below) we know that for lattice point sets in $[0,1)^d$ the separation radius is of the same order of magnitude as the isotropic discrepancy. Hence we have 
\begin{equation}\label{disp_sep}
D_{N_i}^*(P_{N_i}) \le J_{N_i}(P_{N_i}) \asymp h_p(P_{N_i}) \asymp \frac{1}{N_i^{1/d}}.
\end{equation}
This implies that $(P_{N_i})_{i \ge 1}$ is uniformly distributed modulo 1.

\medskip \textbf{A quasi-uniform family of lattice point sets with slowly converging discrepancy.} Consider the regular grid $\Lambda_{m} = \{ (k_1, k_2, \ldots, k_d) / m \ : \ 0 \le k_j < m,\ j= 1, 2, \ldots, d\}$, for $m \in \mathbb{N}$. Then $\{\Lambda_m\}_{m \in \mathbb{N}}$ is a quasi-uniform family, but the star-discrepancy of this point set is only of order \[D_N^{\ast}(\Lambda_m) \asymp \frac{1}{m} = \frac{1}{N^{1/d}},\] where the number of points $|\Lambda_m| = N= m^d$; see, for instance, \cite[Remark~3.33]{DP10}.

\subsection{Motivation and Applications}\label{ssec:motiv}

The mesh-ratio and the related quantities covering radius and separation radius have important practical applications in a number of areas. One such area is experimental designs, including those for computer experiments \cite{FLS06,SWN03}. A central problem in this field is the construction of minimax designs and maximin designs, which aim to minimize the covering radius and to maximize the separation radius, respectively, to ensure the space-filling property; see, e.g., \cite{H21,HRDH11,JMY90,J16,P17,PM12}. 

Another important area is scattered data approximation \cite{SW06,W05}. Recent theoretical advances---particularly in radial basis function approximation and Gaussian process regression---have shown that quasi-uniform point sets (which control both the covering and separation radii) are essential to balance numerical stability with approximation accuracy; see \cite{Sch95,T20,WSH21,WBG21} among others. Typically, for kernel interpolation with standard (isotropic) Sobolev kernels, the theoretical error bound for functions in the native space is primarily determined by the covering radius (the objective of minimax designs). Note that for such isotropic kernels, the convergence rate deteriorates as the dimension $d$ gets larger. However, relying solely on the covering radius is often insufficient. First, a small separation radius leads to an ill-conditioned interpolation matrix (often referred to as the Gram matrix). In practice, this ill-conditioning amplifies numerical round-off errors, which can spoil the theoretical convergence rate and result in a large realized error. Second, in the misspecified case where the target function lies outside the native space, the approximation error is governed by the Lebesgue constant. It is well known that the growth of the Lebesgue constant is sensitive to the point distribution, and a bounded mesh ratio is often required to ensure robust convergence properties \cite{DS10,NWW05,T20,WBG21}. Therefore, simultaneously controlling the covering and separation radii is vital for robust scattered data approximation.

In the context of quasi-Monte Carlo methods, Sobol' and Shukhman~\cite{SS07} studied the minimum distance (separation radius) for the 2-dimensional Sobol' sequence. The first papers on maximized minimum distance in the context of quasi-Monte Carlo rules are \cite{GHSK08} and \cite{GK09}, who focused on $2$ dimensional digital nets with maximized minimum distance for computer graphics applications. The dispersion (covering radius) of quasi-Monte Carlo point sets and applications thereof (quasi-Monte Carlo methods for optimization) is discussed in \cite[Chapter~6]{N92}. 

On the other hand, the need for uniformly distributed point sets and sequences with low discrepancy is well enough known and has its legitimation in the theory of quasi-Monte Carlo methods; see \cite{DKP22,DP10,LePi14,N92}. 

In Section~\ref{ssec:quvsud} we discussed that the two properties, quasi-uniformity and low discrepancy, are not necessarily equivalent. Nevertheless, there may be situations in applications where the presence of both properties is advantageous. We describe one such situation in the following.

Let $K: [0,1]^d \times [0,1]^d \to\mathbb{R}$ be a continuous reproducing kernel and let $\mathcal{H}_K$ be the corresponding reproducing kernel Hilbert space (see \cite{Aron50} for information about the theory of reproducing kernel Hilbert spaces). Let $f \in \mathcal{H}_K$ and consider the problem of approximating the function $f$ given only function values. One method of approximating $f$ is to use kernel interpolation, that is, for a given point set $P_N = \{ \boldsymbol{x}_0, \boldsymbol{x}_1, \ldots, \boldsymbol{x}_{N-1} \} \subseteq [0,1]^d$, we use
\begin{equation*}
f(\boldsymbol{x}) \approx \mathcal{A}(f)(\boldsymbol{x}) := \sum_{n=0}^{N-1} c_n K(\boldsymbol{x}, \boldsymbol{x}_n)
\end{equation*}
with coefficients $c_0,c_1,\ldots,c_{N-1}$ that are determined in the following way: By demanding that $f(\boldsymbol{x}_m) = \mathcal{A}(f)(\boldsymbol{x}_m)$ for all $m \in \{0, 1, \ldots, N-1\}$, we obtain the linear system
\begin{equation*}
f(\boldsymbol{x}_m) = \sum_{n=0}^{N-1} c_n K(\boldsymbol{x}_m, \boldsymbol{x}_n)
\end{equation*}
in the unknowns $c_0, c_1, \ldots, c_{N-1}$. We can write this in matrix form as
\begin{equation}\label{linsysK}
\boldsymbol{f} = \mathcal{K} \boldsymbol{c},
\end{equation}
where $\boldsymbol{f} = (f(\boldsymbol{x}_0), f(\boldsymbol{x}_1), \ldots, f(\boldsymbol{x}_{N-1}))^\top$, $\mathcal{K} = [K(\boldsymbol{x}_n, \boldsymbol{x}_m)]_{n,m=0}^{N-1}$, and $\boldsymbol{c} = (c_0, c_1,\ldots, c_{N-1})^\top$.

In order to obtain a small approximation error, one approach is to use a point set $P_N$ as nodes of the involved integration rule for which the numerical integration error in some function space is small (see \cite[Section~3.2.1]{ZKH09}). Depending on the function space, the integration error can often be related to the discrepancy of the underlying point set. Further, it is known from \cite[Chapter~6]{N92} that many quasi-Monte Carlo point sets with small discrepancy have a covering radius which is asymptotically optimal.

However, such asymptotic optimality of the covering radius does not preclude the existence of point pairs with an arbitrarily small distance. This lack of separation becomes critical in the context of approximation. Consider the case where the point set $P_N$ has two points $\boldsymbol{x}_i, \boldsymbol{x}_j$, which are very close to each other. Then the corresponding rows of the matrix $\mathcal{K}$ are almost the same and the linear system \eqref{linsysK} is ill-conditioned. Hence, in this situation, it is advantageous to use point sets that are not only low in discrepancy (or have small worst-case integration error) but are also quasi-uniform. Moreover, such quasi-uniformity ensures numerical stability and robust convergence for scattered data approximation with standard Sobolev kernels, not only within the native space but also in the misspecified setting.

\subsection{Organization of the paper}

The primal objects of interest in this work are lattice point sets and, closely related, Kronecker sequences (also referred to as $(n \boldsymbol{\alpha})$-sequences). In Section~\ref{sec:lat}, we recall the basic definition of lattices and lattice point sets and provide some tools and lemmas regarding lattices and mesh ratio. 

The main results of this work will be presented in Section~\ref{sec:main}. In Section~\ref{sec:shsradlat}, Theorem~\ref{thm:admislat}, we present a construction of quasi-uniform nested shrunk and shifted lattices of low star-discrepancy by means of admissible lattices. Section~\ref{sec:r1lat} deals with rank-1 lattice point sets. In Corollary~\ref{cor:meshratio-rank1lattice} we present an explicit construction of a rank-1 lattice point set in dimension two with bounded mesh ratio and low star-discrepancy. In Theorem~\ref{thm:exlatd} we consider the $d$-dimensional case and provide an existence result for rank-1 lattice point sets with bounded mesh ratio and low star-discrepancy. Finally, in Section~\ref{secc:nalpha} we consider $(n \boldsymbol{\alpha})$-sequences and characterize exactly those $\boldsymbol{\alpha} \in \RR^d$ for which the corresponding sequence is quasi-uniform; see Theorem~\ref{thm:nalpha}. In a last step, we show that some of them even have very low star-discrepancy.

\medskip

\section{Background on lattices}\label{sec:lat}

In this work we put the main focus on lattice point sets and on $(n \boldsymbol{\alpha})$-sequences. 

\begin{definition}[Lattices]
Let $T\in \RR^{d\times d}$ be an invertible matrix. This matrix generates a {\it lattice} in $\RR^d$ given by
\[ \XX=T(\ZZ^d)=\{ T\bsk \ : \ \bsk\in \ZZ^d\}. \]
For later use we introduce also the {\it dual lattice} of $\XX$ which is $$\XX^{\bot}=\{ \bsh \in \RR^d \ : \ \bsh \cdot \bsy \in \ZZ \mbox{ for all } \bsy \in \XX\},$$ where $\cdot$ denotes the standard inner product on the $\RR^d$.
\end{definition}

The separation radius of a lattice $\XX$ over $\RR^d$ satisfies \[ q_{p}(\XX; \RR^d)
=\min_{\substack{\bsx,\bsy\in \XX \\ \bsx\neq \bsy}}\dfrac{\|\bsx-\bsy\|_p}{2}
= \min_{\bsx \in \XX \setminus \{\bszero\}} \dfrac{\|\bsx\|_p}{2}
\]
and we define the covering radius by $h_p(\XX; \RR^d) = \sup_{\bsx \in \RR^d} \min_{\bsy \in \XX} \|\bsx-\bsy\|_p$.

Point sets with a lattice structure are often used as quadrature nodes in quasi-Monte Carlo algorithms; see, e.g., \cite{DKP22,N92,SJ94}. Due to the repeating structure, lattices are also the leading candidates for point sets with an optimal mesh ratio. Before studying the mesh ratio of some lattices, which are also of use in quasi-Monte Carlo algorithms, we first provide some relevant background material which will be used at later stages.

\medskip \textbf{Successive minima.} The covering radius, separation radius and mesh ratio of a lattice are closely related to the notion of successive minima.
For $i \in \{1,\ldots,d\}$, the $i$-th {\it successive minima} $\lambda_i^{(p)}$ (in $\ell_p$-norm) of a lattice $\XX$ is defined as the minimum of all positive reals $\lambda$ such that $\XX$ contains at least $i$ linearly independent vectors whose $\ell_p$-norms are at most $\lambda$. A vector which attains $\lambda_1^{(p)}$ is called a {\it shortest vector} (in $\ell_p$-norm).
By definition, we have 
\begin{equation}\label{eq:lat-sep}
q_p(\XX; \RR^d) = \frac{\lambda_1^{(p)}}{2}.
\end{equation}
Furthermore, we have a bound on the covering radius as follows.

\begin{lemma}\label{lem:cov-by-minima}
For a lattice $\XX\subseteq \RR^d$, we have
\begin{equation} \label{eq:lat-cov}
h_{p}(\XX;\RR^d) \le \dfrac{1}{2}\sum_{j=1}^d \lambda_j^{(p)} \le \dfrac{d}{2}\lambda_d^{(p)}.   
\end{equation}
\end{lemma}

\begin{proof}
By the definition of the successive minima,
we have linearly independent vectors $v_1,\dots,v_{d} \in \XX$ with $\|v_j\|_p \le \lambda_j^{(p)}$ for $j \in \{1, 2, \ldots, d\}$. Note that the lattice generated by these vectors is generally a sublattice of $\XX$.

Let $\bsx \in \RR^d$ and let $c_1,\dots,c_d \in \RR$ with $\bsx = \sum_{j=1}^d c_j v_j$.
Let $c'_j$ be the nearest integer to $c_j$ and let $\bsy := \sum_{j=1}^d c'_j v_j \in \XX$.
Then we have
\[
\|\bsx - \bsy\|_p
= \left\|\sum_{j=1}^d(c'_j-c_j)v_j\right\|_p
\le \sum_{j=1}^d |c'_j-c_j|\|v_j\|_p
\le \frac{1}{2} \sum_{j=1}^d \lambda_j^{(p)}.
\]
Thus, the first inequality of this lemma follows.
The second inequality is obvious since we have $\lambda_1^{(p)} \le \dots \le \lambda_d^{(p)}$.
\end{proof}

We now show that the mesh ratio of the lattices from a family of lattices with relatively large shortest vectors is bounded by a constant, concluding that such lattices are quasi-uniform.

\begin{lemma}\label{lem:minima_meshratio}
Let $\XX=T(\ZZ^d)$ be a lattice such that $\lambda_1^{(\infty)} \ge C (\det T)^{1/d}$ for some $C>0$.
Then we have
\[
\lambda_d^{(\infty)} \le \frac{(\det T)^{1/d}}{C^{d-1}}.
\]
Moreover, it follows that
\[
\rho_\infty(\XX;\RR^d)
\le \dfrac{d}{C^d}.
\]
\end{lemma}

\begin{proof}
Minkowski's second theorem (see, for example, \cite[Theorem~16]{Siegel89} or \cite[Appendix~B]{C57}), states that
\[
\prod_{j=1}^d \lambda_j^{(\infty)} \le \det T.
\]
Combining this with the inequality $C(\det T)^{1/d} \le \lambda_1^{(\infty)} \le \lambda_j^{(\infty)}$ for $j \in \{2,\dots, d-1\}$, we obtain the first claim.
The second claim follows from \eqref{eq:lat-sep}, \eqref{eq:lat-cov} and the first claim as
\[
\rho_\infty(\XX;\RR^d) = \dfrac{h_\infty(\XX; \RR^d)}{q_\infty(\XX; \RR^d)}
\le \dfrac{d \; \lambda_d^{(\infty)}/2}{\lambda_1^{(\infty)}/2}
\le \dfrac{d}{C^d}. 
\qedhere
\]
\end{proof}

\medskip \textbf{Lattice points inside the hypercube.}
Let us now turn to the quasi-uniformity of lattice points in the unit hypercube. Specifically, for a $d$-dimensional lattice $\XX = T(\ZZ^d)$, we consider the point set $P := \XX \cap [0,1)^d$ and call $P=P(\XX)$ a lattice point set (in $[0,1)^d$). In general, this restriction changes the covering radius and the separation radius, but the impact is relatively small as the following lemma shows.

\begin{lemma}\label{le3.3}
Let $\XX = T(\ZZ^d)$ be a $d$-dimensional lattice and $P = \XX \cap [0,1)^d$. Assume that $|P|\geq 2$.
Then, for any $p \in [1,\infty]$, we have
\begin{equation}\label{bound_hq_lattice}
q_p(P) \ge q_p(\XX;\RR^d) \qquad \mbox{ and } \qquad h_\infty(P) \le 2 h_\infty(\XX;\RR^d).
\end{equation}
In particular, 
\begin{equation}\label{eq:meshratio-restriction}
\rho_\infty(P) = \dfrac{h_\infty(P)}{q_\infty(P)} \le 2\dfrac{h_\infty(\XX;\RR^d)}{q_\infty(\XX;\RR^d)} = 2\rho_\infty(\XX;\RR^d).
\end{equation}
\end{lemma}

\begin{proof}
The result for the separation radius follows from its definition, since
\[
q_{p}(P) 
=\min_{\substack{\bsx,\bsy\in P \\ \bsx\neq \bsy}}\dfrac{\|\bsx-\bsy\|_p}{2}
\ge \min_{\substack{\bsx,\bsy\in \XX \\ \bsx\neq \bsy}}\dfrac{\|\bsx-\bsy\|_p}{2}
= q_{p}(\XX; \RR^d).
\]

Let us move on to the covering radius.
If $2 h_\infty(\XX;\RR^d) \ge 1$, then the result is clear, since $\bszero \in P$ and hence $$h_{\infty}(P) \le \sup_{\bsx \in [0,1]^d} \|\bsx\|_{\infty} =1 \le 2 h_\infty(\XX;\RR^d).$$ 

Next, consider the case $2 h_\infty(\XX;\RR^d) < 1$. 
For $\bsx = (x_1,\dots,x_d) \in [0,1)^d$ put $x_j^* := \max(0,x_j-2 h_\infty(\XX;\RR^d))$ for $j \in \{1,\dots,d\}$ and define the hypercube $D$ in $[0,1)^d$ by
\[
D := \prod_{j=1}^d [x_j^*,x_j^*+2 h_\infty(\XX; \RR^d)].
\]
Let $\bsz$ be the center of $D$. By the definition of the covering radius, there exists $\bsy^* \in \XX$ such that
\[
\|\bsy^*-\bsz\|_\infty \le h_\infty(\XX;\RR^d). 
\]
This implies $\bsy^* \in D$ and thus $\bsy^* \in P$.
Hence we have
\[
\min_{\bsy\in P} \|\bsx-\bsy\|_\infty
\le \|\bsx-\bsy^*\|_\infty
\le 2h_\infty(\XX;\RR^d). 
\]
Since this bound is independent of the location of $\bsx \in [0,1)^d$, we obtain the desired result.

The result \eqref{eq:meshratio-restriction} is clear.
\end{proof}

\begin{remark}
Although we omit the details, the result of the lemma holds more generally for any hypercube of the form $\Omega=[a_1,b_1)\times \cdots \times [a_b,b_d)$ with $a_j<b_j$ for all $j \in \{1,\ldots,d\}$. More precisely, with $P_{\Omega}= \XX \cap \Omega$ such that $|P_\Omega| \ge 2$, for any $p \in [1,\infty]$ we have $$
q_p(P_{\Omega}; \Omega) \ge q_p(\XX;\RR^d) \qquad \text{ and } \qquad h_\infty(P_{\Omega}; \Omega) \le 2 h_\infty(\XX;\RR^d).$$
This implies that any shifted lattice in the unit hypercube, defined by
\[ P=\left\{ T\bsk-\bsdelta  \ : \ \bsk\in \ZZ^d \right\}\cap [0,1)^d\quad \text{with $\bsdelta\in \RR^d$,}\]
also satisfies the same bound \eqref{eq:meshratio-restriction} on the mesh ratio.
\end{remark}

With \eqref{eq:meshratio-restriction}, we can show that the ``shrinking'' strategy works to construct quasi-uniform point sets and sequences.

\begin{lemma}\label{lem_ps_lattice}
Let $\XX = T(\ZZ^d)$ be a $d$-dimensional Euclidean lattice. For $a>0$, define
\[ P_a:=a^{-1}\XX \cap [0,1)^d, \] and assume that $|P_a|\ge 2$. Then, it holds that \[ \rho_\infty(P_a) \leq 2\rho_\infty(\XX;\RR^d)\quad \mbox{for every } a>0 \mbox{ such that } |P_a| \ge 2.  \]
In particular, for any set of positive integers $X$,
$\{P_{a}\}_{a \in X}$ is a quasi-uniform family.
\end{lemma}

\begin{proof}
It is evident that $a^{-1} \XX$ is a $d$-dimensional lattice and satisfies
\begin{equation}\label{eq_lattice_hq}
h_{p}(a^{-1} \XX; \RR^d) = a^{-1} h_{p}(\XX; \RR^d) \quad \text{and} \quad q_{p}(a^{-1} \XX; \RR^d) = a^{-1} q_{p}(\XX; \RR^d),
\end{equation}
for all $p \in [1,\infty]$.
From \eqref{eq:meshratio-restriction}, it follows that
\[
\rho_\infty(P_a) \leq 2 \frac{h_\infty(a^{-1} \XX; \RR^d)}{q_\infty(a^{-1} \XX; \RR^d)} = 2 \frac{a^{-1} h_\infty(\XX; \RR^d)}{a^{-1} q_\infty(\XX; \RR^d)} = 2 \rho_\infty(\XX; \RR^d).
\]
Since this bound is independent of $a$, the claim follows.
\end{proof}

\begin{remark}\label{rem_growth_cond}
Let $\XX = T(\mathbb{Z}^d)$ be a lattice with $\rho_\infty(\XX; \mathbb{R}^d) < \infty$.
For $a \in \NN$,
let $P_{a} = a^{-1} \XX \cap [0,1)^d$ and $N_{a} = |P_{a}|$. The covering radius for any finite point set $P$ satisfies $h_\infty(P) \gg |P|^{-1/d}$ and the separation radius satisfies $\rho_\infty(P) \ll |P|^{-1/d}$. Using this together with \eqref{bound_hq_lattice} and \eqref{eq_lattice_hq} we obtain
\begin{equation*}
N_{a}^{-1/d} \ll h_\infty(P_{a}) \le 2 a^{-1} h_\infty(\XX).
\end{equation*}
This implies that $N_{a} \gg a^d$. Using the separation radius we similarly obtain $N_{a} \ll a^d$, and thus $N_a \asymp a^d$. 
\end{remark}

\begin{remark}\label{rem_extensible_lattices}
Let $n  > 1$ be an integer, and consider the case where $a_i = a_0 n^i$ for all $i \in \NN$ with some $a_0>0$. We see that 
    \begin{align*}
        a_{i+1}^{-1} \XX 
        & = \left\{ a_0^{-1} n^{-i-1} T \bsk \ : \  \bsk \in \ZZ^d \right\} \\
        & = \left\{ a_0^{-1} n^{-i-1} T \bsk \ : \  \bsk \in (n \ZZ)^d \right\} \cup \left\{ a_0^{-1} n^{-i-1} T \bsk \ : \ \bsk \in \ZZ^d \setminus (n \ZZ)^d \right\} \\
        & = \left\{ a_0^{-1} n^{-i} T \bsk \ : \  \bsk \in \ZZ^d \right\} \cup \left\{ a_0^{-1} n^{-i-1} T \bsk \ : \  \bsk \in \ZZ^d \setminus (n \ZZ)^d \right\} \\
        & = a_i^{-1} \XX \cup \left\{ a_0^{-1} n^{-i-1} T \bsk \ : \  \bsk \in \ZZ^d \setminus (n \ZZ)^d \right\} \supseteq a_i^{-1} \XX.
    \end{align*}
Hence, the sequence $(a_i^{-1} \XX)_{i \ge 1}$ is a nested sequence (in quasi-Monte Carlo theory this property is referred to as ``extensibility'')
and therefore the set $\{P_{a_i}\}_{i \ge 1}$, where $P_{a_i}=a_i^{-1}\XX\cap[0,1)^d$, forms a quasi-uniform family of nested point sets.
Since $|P_{a_i}| \asymp |a_i|^d$, we have $\sup_{i\ge1} |P_{a_{i+1}}|/|P_{a_i}|<\infty$, i.e., the cardinality of the sets grows at most geometrically.
In particular, by ordering the newly added points at each extension step, one may regard $\{P_{a_i}\}_{i \ge 1}$ as the sequence of initial segments of a single infinite point sequence (and we will freely switch between the “nested family” and “sequence” viewpoints when convenient).
From Lemma~\ref{lem_seq_weak}, this sequence is quasi-uniform.
\end{remark}

\medskip \textbf{The spectral test and enhanced trigonometric degree of lattices.} The following result can be easily brought together from well-known results in the literature. Let $\XX = T(\ZZ^d)$ be a $d$-dimensional lattice and $P = \XX \cap [0,1)^d$ with $|P|=N$. The spectral-test of the lattice $\XX$, is $$\sigma_N(P)=\frac{1}{\min_{\bsh \in \XX^{\bot}\setminus \{\bszero\}} \|\bsh\|_2}$$ and the enhanced trigonometric degree of the lattice $\XX$ is $$\kappa_N(P) = \min_{\bsh \in \XX^{\bot}\setminus \{\bszero\}} \|\bsh\|_1.$$
Like discrepancy, the spectral-test and enhanced trigonometric degree are important figure-of-merits for lattices in the context of numerical integration. For general information we refer to, e.g., \cite[Section~1.8]{DKP22}.

\begin{lemma}\label{le:equiv}
Let $d \in \mathbb{N}$ and $\XX = T(\ZZ^d)$ be a $d$-dimensional lattice and $P = \XX \cap [0,1)^d$ with $|P|=N$. Then for any $p \in [1,\infty]$ we have 
$$h_p(P) \asymp_{p,d} J_N(P) \asymp_{p,d} \sigma_N(P) \asymp_{p,d} \frac{1}{\kappa_N(P)},$$
where $J_N(P)$ denotes the isotropic discrepancy of $P$, $\sigma_N(P)$ is the spectral-test of the lattice $\XX$ and $\kappa_N(P)$ is the enhanced trigonometric degree of the lattice $\XX$.
\end{lemma}

\begin{proof}
The equivalence $h_p(P) \asymp_{p,d} J_N(P)$ is \cite[Proposition~8]{SP21} (with $\gamma=\infty$), the equivalence $J_N(P) \asymp_{p,d} \sigma_N(P)$ is \cite[Theorem~2]{PS20}, and the equivalence $\sigma_N(P) \asymp_{p,d} 1/\kappa_N(P)$ is clear from the definition and the equivalence of $\ell_p$-norms. 
\end{proof}

In the next sections we provide examples of lattice constructions in the unit cube which have both, small discrepancy and bounded mesh ratio. 

\section{The main results}\label{sec:main}

\subsection{Admissible lattices - Frolov's construction}\label{sec:shsradlat}

Let $\XX$ be an admissible lattice, that is, $\XX$ satisfies \[ \inf_{\bsx\in \XX\setminus \{\bszero\}} \prod_{j=1}^{d}\left|x_j\right|>0. \] An example of such a lattice was proposed by Frolov \cite{F76} (see also \cite{U16}). The discrepancy of admissible lattices was studied in \cite{F80,S94}. 

Let $\XX=T(\ZZ^d)$ be an admissible lattice. From \cite[Corollary~2.1, item~1]{S94} it is known that for any shift $\bsdelta \in \RR^d$ and any $a > 0$ the shrunk, shifted lattice $P$ inside the hypercube of the form 
\begin{equation}\label{def:shsradlat}
P_{a,\bsdelta} = \left\{ a^{-1}\left(T\bsk-\bsdelta\right) \ : \ \bsk\in \ZZ^d\right\}\cap [0,1)^d 
\end{equation} 
has star-discrepancy of order of magnitude $(\log N)^{d-1}/N$, where $N=|P_{a,\bsdelta}|$.

Using this result, Remark~\ref{rem_growth_cond} and Remark~\ref{rem_extensible_lattices}, we can obtain a family of nested shifted lattice point sets with very low (possibly best possible) star-discrepancy, which is also quasi-uniform.

\begin{theorem}\label{thm:admislat}
Let $\XX=T(\ZZ^d)$ be an admissible lattice.  For $a>0$ and $\bsdelta \in \RR^d$ let $P_{a,\bsdelta}$ be the point set from \eqref{def:shsradlat}. For an integer $n \geq 2$, let $a_i = a_0 n^i$ for all $i \in \NN$ with some $a_0>0$ and let $\bsdelta_i \in \RR^d$ for all $i \in \NN$. Then the set $\{P_{a_i,\bsdelta_i}\}_{i \ge 1}$ is a quasi-uniform family of point sets whose star-discrepancy is of order of magnitude $(\log N_i)^{d-1}/N_i$, where $N_i=|P_{a_i,\bsdelta_i}|$ for all $i \in \NN$.
Furthermore, if we take a common shift $\bsdelta \in \RR^d$ and set $\bsdelta_i=\bsdelta$ for all $i$, the family is nested.
\end{theorem}

For practical implementations of such lattices, see \cite{C24,KOUU21,SY19}.

\subsection{Rank-1 lattice point sets}\label{sec:r1lat}

An important instance of lattice point sets are so-called rank-1 lattice point sets. Let $N \in \NN$ and $\bsg =(g_1,\ldots,g_d)\in \{1, \ldots, N-1\}^d$ with $\gcd(g_i,N)=1$ for all $i \in \{1,\ldots,d\}$.
Let $\XX(\bsg,N) = \{ T \bsk \ : \ \bsk \in \mathbb{Z}^d \}$ where 
\begin{equation}\label{eq:r1lat}
T = \begin{pmatrix} 1/N & 0 & \ldots & \ldots  & \ldots & 0 \\ g_1^{-1} g_2/N & 1 & 0 & \ldots & \ldots & 0 \\ \vdots & 0 & \ddots & \ddots &  & \vdots \\ \vdots & \vdots & \ddots & \ddots & \ddots & \vdots \\ \vdots & \vdots &  & \ddots & \ddots & 0 \\ g_1^{-1} g_d/N & 0 & \ldots & \ldots & 0 & 1 \end{pmatrix}.
\end{equation}
be a so-called {\it rank-1 lattice} and $P(\bsg,N)=\XX(\bsg,N) \cap [0,1)^d$, i.e. 
\[P(\bsg,N)=\left\{ \left\{ \frac{k}{N} \bsg\right\}\ : \ k \in \{0,1,\ldots,N-1\}\right\}\]
be the corresponding {\it rank-1 lattice point set} consisting of $N$ points in $[0,1)^d$. In \eqref{eq:r1lat}, $g_1^{-1}$ denotes the multiplicative inverse of $g_1$ modulo $N$. In the present case, the dual lattice is given by
\[\XX^{\bot}(\bsg,N) = \{\bsh \in \ZZ^d \ : \ \bsg \cdot \bsh \equiv 0 \pmod{N}\},\] see \cite[Section~1.2]{DKP22}. Detailed expositions on rank-1 lattice point sets in the theory of quasi-Monte Carlo integration can be found in the books \cite{DKP22, N92, SJ94}.

In dimension $2$ we can obtain explicit constructions of rank-1 lattice point sets with small discrepancy, which are quasi-uniform. However, in higher dimensions, we can only prove an existence result.

\subsubsection{Rank-1 lattice point sets in dimension two}
We study the quasi-uniformity of rank-1 lattice point sets $P((1,g), N) \subseteq [0,1)^2$ for certain $g \in \{1, 2, \ldots, N-1\}$. In order to study the covering and separation radii of such lattice point sets, we need to bound the successive minima; see Section~\ref{sec:lat}. The candidates of a shortest vector of a planar lattice generated by an integer matrix are given by the continued fraction of $g/N$; see \cite{eisenbrand2001short}. Considering all candidates, we can show a bound on the mesh ratio.
For the sake of completeness, we provide a full proof. Before we do so, we provide some background on continued fraction expansions.

A continued fraction represents a positive real number $x$ in the form
\[
x = a_0 + \cfrac{1}{a_1 + \cfrac{1}{a_2 + \cfrac{1}{a_3 + \cfrac{1}{\ddots}}}}
=: [a_0;a_1,a_2,a_3,\dots]
\]
where $a_0, a_1, a_2, a_3, \dots$ are positive integers.
The $n$-th convergent of $x$ is given by $r_n = [a_0;a_1,\dots,a_n] = p_n/q_n$ with $\gcd(p_n,q_n)=1$, where $p_n$ and $q_n$ can be computed recursively 
\begin{align*}
p_n &= a_n p_{n-1} + p_{n-2},\\
q_n &= a_n q_{n-1} + q_{n-2},
\end{align*}
with initial conditions
\[
p_{-2} = 0, \, p_{-1} = 1, q_{-2} = 1, \, q_{-1} = 0.
\]
If $x$ is rational, the continued fraction expansion is finite and uniquely determined if we write it in the form
\begin{equation}\label{eq:cf-rational}
x = [a_0;a_1,\dots,a_\ell] \qquad \text{with } a_\ell=1\,\text{ for } \ell \in \NN.
\end{equation}
The convergents are good Diophantine approximations of $x$. That is, for any real $x$ the $n$-th convergent $r_n = p_n/q_n$ satisfies
\begin{equation}\label{eq:error-of-convergent}
\frac{1}{q_n (q_n+q_{n+1})} \le \left| x - \frac{p_n}{q_n} \right| \le \frac{1}{q_n q_{n+1}}.    
\end{equation}
Furthermore, the $n$-th convergent is a best rational approximation of $x$ in the sense that,
for any integers $1 \le h < q_{n+1}$ and $b$ we have
\begin{equation}\label{eq:best-approx}
|hx-b| \ge |q_n x - p_n|.
\end{equation}
For more information on continued fraction expansions and their use in Diophantine approximation theory, we refer to \cite{Khi}.

Now we return to the mesh-ratio of rank-1 lattice point sets in dimension 2.

\begin{theorem}\label{meshratio-by-cf}
Let $g, N \in \NN$ with $\gcd(g,N)=1$ and assume that $x = g/N$ has continued fraction expansion of the form \eqref{eq:cf-rational}. Let
$$T = \begin{pmatrix} 1/N & 0 \\ g/N & 1 \end{pmatrix} \qquad \mbox{ and } \qquad \XX = \XX((1,g),N)= T(\ZZ^2).$$
Then, denoting $K = \max(a_1,\dots,a_\ell)$, we have
\[
\lambda_1^{(\infty)} \ge \dfrac{1}{\sqrt{(K+2)N}}.
\]
Furthermore, we have
\[
\rho_\infty(\XX;\RR^2) \le 2(K+2).
\]
\end{theorem}

\begin{proof}
To show the first claim, it suffices to show that $\|\bsx\|_\infty \ge 1/\sqrt{(K+2)N}$ holds true for any $\bsx \in \XX \setminus \{\bszero\}$. Now every $\bsx \in \XX \setminus \{\bszero\}$ can be written as $(t/N,tg/N-u)$ for some $t,u \in \ZZ$ with $(t,u) \neq (0,0)$. By symmetry, we only need to consider the case $t \ge 0$. We distinguish three cases:
\begin{itemize}[leftmargin = *]
\item If $t=0$, then $u \neq 0$ and thus we have
$
\|(t/N,tg/N-u)\|_{\infty} = |u| \ge 1 \ge 1/\sqrt{(K+2)N}.
$
\item If $t \ge N$, then we have
$\|(t/N,tg/N-u)\|_{\infty} \ge t/N \ge 1 \ge 1/\sqrt{(K+2)N}$.
\item Now let $0 < t < N$. Let $p_n/q_n$ be the $n$-th convergent of $x=g/N$ for $n=0,\dots,\ell$. Take $j$ with $q_j \le t < q_{j+1}$.
Then it follows from \eqref{eq:best-approx} and \eqref{eq:error-of-convergent} that
\begin{align*}
\left|\dfrac{tg}{N}-u\right|
\ge & \left|\dfrac{q_jg}{N}-p_j\right|
= q_j\left|\dfrac{g}{N}-\dfrac{p_j}{q_j}\right|
\ge \dfrac{1}{q_j+q_{j+1}}\\
= & \dfrac{1}{q_j+(a_{j+1}q_j+q_{j-1})}
\ge \dfrac{1}{(K+2)q_j}.
\end{align*}
Thus we have
\[
\left\|\left(\dfrac{t}{N},\dfrac{tg}{N}-u\right)\right\|_\infty
\ge \max\left(\dfrac{q_j}{N},\dfrac{1}{(K+2)q_j}\right)
\ge \dfrac{1}{\sqrt{(K+2)N}},
\]
where we use $\max(a,b) \ge \sqrt{ab}$ for $a,b>0$ in the last inequality.
\end{itemize}
These three cases imply the first claim. Noting that $\det T = 1/N$, the second claim follows directly from the first one and Lemma~\ref{lem:minima_meshratio}.
\end{proof}

Combining Theorem~\ref{meshratio-by-cf} and \eqref{eq:meshratio-restriction},
we obtain the following corollary:

\begin{corollary}\label{cor:meshratio-rank1lattice}
Let $g,N \in \NN$ with $\gcd(g,N)=1$ and $P = P((1,g),N)$ be the corresponding rank-1 lattice point set. Assume $x = g/N$ has continued fraction expansion of the form \eqref{eq:cf-rational} and $K := \max(a_1,\dots,a_\ell)$. Then we have
\[
\rho_\infty(P) \le 4(K+2).
\]
\end{corollary}

\begin{remark}
It is well-known that the discrepancy of $2$-dimensional rank-$1$ lattice point sets is also strongly related to the continued fraction expansion of $g/N$.
Under the same assumption of Corollary~\ref{cor:meshratio-rank1lattice},
the star-discrepancy of $P((1,g),N)$ satisfies
\[
D_N^*(P((1,g),N)) \ll K \frac{\log N}{N}
\]
with an absolute implied constant \cite[(5.39) and the following discussion]{N92}. Thus such point sets are also low-discrepancy point sets.
\end{remark}

\begin{remark} 
Let $F_1 = F_2 = 1$ and $F_k = F_{k-1} + F_{k-2}$ for $k \ge 3$ be the Fibonacci numbers. Then the Fibonacci lattice point set is given by
\begin{equation*}
\mathcal{F}_m = \left\{ \left(\frac{n}{F_m}, \left\{\frac{n F_{m-1}}{F_m} \right\} \right) \ : \  n = 0, 1, 2, \ldots, F_m-1 \right\},
\end{equation*}
i.e., in general notation, $\mathcal{F}_m =P((1,g),N)$ with $g=F_{m-1}$ and $N=F_m$.

Since it is well-known that  $F_{m-1}/F_m = [0;1,1,\dots,1]$,
it follows from Theorem~\ref{meshratio-by-cf} with $K=1$ and $F_{m+1} \le 2 F_{m}$ that the family of Fibonacci lattices is quasi-uniform and also of low-discrepancy.
Since the explicit shortest basis of (the underlying planar lattice of) Fibonacci lattices is known (see \cite{ABD12,NS94}), one can obtain the exact separation radius and tighter bounds on the covering radius and the mesh ratio of the Fibonacci lattices.
\end{remark}

\begin{remark}
The minimum distance of rank-1 lattice point sets has been studied in the context of the traveling salesman path (see \cite{pausinger2017bounds,pausinger2017lattices}). Consider the family of lattices $P_k = P((1,g_k), N_k)$ where
$g_k/N_k = [0;a_1,\cdots,a_{2k+1}]$ with $a_{2i}=1$ and $a_{2i+1}=2$ holds.
It is shown in \cite{pausinger2017lattices} that the shape of the lattices $P_k$ converges to the hexagonal lattice and the minimum distance of $P_k$ converges to $\sqrt{2/(\sqrt{3}N)}$. This family is also of low-discrepancy and, according to Corollary~\ref{cor:meshratio-rank1lattice}, also quasi-uniform.
\end{remark}

\subsubsection{Rank-1 lattice point sets in arbitrary dimension $d$}

We now turn the case of general dimension $d \in \NN$ and provide an existence result of rank-1 lattice point sets with bounded mesh ratio and low star-discrepancy.

According to Lemma~\ref{le:equiv}, the covering radius of $P(\bsg,N)$ satisfies $$h_p(P(\bsg,N)) \ll_{d,p} \sigma(P(\bsg,N)) = \frac{1}{\min_{\bsh \in \XX^{\bot}(\bsg,N)\setminus \{\bszero\}} \|\bsh\|_2},$$ where $\sigma(P(\bsg,N))$ is the spectral test of $\XX(\bsg,N)$. Furthermore, the separation distance is $$q_p(P(\bsg,N))=\frac{1}{2} \min_{\bsx \in P(\bsg,N)\setminus \{\bszero\}} \|\bsx\|_p\ge \frac{1}{2}\min_{\bsx \in \XX(\bsg,N)\setminus \{\bszero\}} \|\bsx\|_p .$$ 

These considerations imply 
$$\rho_p(P(\bsg,N)) \ll_{d,p} \frac{1}{\min_{\bsh \in \XX^{\bot}(\bsg,N)\setminus \{\bszero\}} \|\bsh\|_p \min_{\bsh \in \XX(\bsg,N)\setminus \{\bszero\}} \|\bsh\|_p}.$$
Write $$\kappa(\XX)=\min_{\bsh \in \XX\setminus\{\bszero\}}\|\bsh\|_1.$$ Then we have 
\begin{equation}\label{est:rho:kappa}
\rho_p(P(\bsg,N)) \ll_{d,p} \frac{1}{\kappa(\XX(\bsg,N)) \kappa(\XX^{\bot}(\bsg,N))}.
\end{equation}

The following theorem shows the existence of rank-1 lattice point sets with bounded mesh ratio and low star-discrepancy.

\begin{theorem}\label{thm:exlatd}
For every dimension $d \in \NN$, every $p \in [1,\infty]$ and every $\delta \in (0,1)$ there exist quantities $B_{d,p, \delta},C_ {d,p, \delta}>0$, which only depend on $d$, $p$ and $\delta$, with the following property: for every prime number $N$, there exist at least $\lceil \delta (N-1)^d \rceil$ vectors $\bsg \in \{1,\ldots,N-1\}^d$ such that the mesh ratio is bounded $$ \rho_p(P(\bsg,N)) \le B_{d,p,\delta},$$ and the star-discrepancy satisfies $$D_N^{\ast}(P(\bsg,N)) \le C_{d,p,\delta} \frac{(\log N)^d}{N}.$$ 
\end{theorem}

\begin{proof}
Let $\delta \in (0,1)$. It follows from a standard averaging argument, see, e.g., \cite[Lemma~2.13 and Sec.~5.1] {DKP22}, that there exists a $C'_{d,\delta}>0$ with the property, that 
\begin{equation}\label{est:disc}
D_N^{\ast}(P(\bsg,N)) \le C_{d,\delta}' \frac{(\log N)^d}{N} 
\end{equation}
for at least $((2+\delta)/3) (N-1)^d$ generating vectors $\bsg \in \{1,\ldots,N-1\}^d$.

Now we consider the existence of rank-1 lattice point sets with bounded mesh ratio. From \eqref{est:rho:kappa} it follows that we need prove bounds on $\kappa(\XX(\bsg, N))$ and $\kappa(\XX^\perp(\bsg, N))$.

First consider $\kappa(\XX^{\bot}(\bsg,N))$. The proof follows \cite[Proof of Proposition~1.61]{DKP22}, for completeness we include it here.

For a given $\bsh=(h_1,\ldots,h_d) \in \ZZ^d\setminus\{\bszero\}$ with $|h_j|<N$ for all $j \in \{1,\dots,d\}$, 
since we assume that $N$ is a prime number, there are $(N-1)^{d-1}$ choices of $\bsg \in \{1,\ldots, N-1\}^d$
such that $\bsh \cdot \bsg \equiv 0\pmod{N}$.  Furthermore
\begin{equation*}
|\{\bsh \in \ZZ^d \ :\ \|\bsh\|_1 = \ell \}| \le 2^d \binom{\ell + d-1}{d-1}.
\end{equation*}
Let $\kappa<N$ be a given positive integer (note that we always have $\kappa(\XX^{\bot}(\bsg,N))<N$). Put $$A_{\kappa}:=\{\bsh \in \ZZ^d \ : \ \|\bsh\|_1 \le \kappa \}.$$ Then
\begin{equation*}
|A_\kappa| \le
2^d \sum_{\ell = 0}^\kappa \binom{\ell + d-1}{d-1} = 2^d \binom{\kappa+d}{d}.
\end{equation*}
As stated above, for every $\bsh$ in the set $\{\bsh \in \ZZ^d \ : \ \|\bsh\|_1 \le \kappa\}$ there are $(N-1)^{d-1}$ choices 
$\bsg \in \{0,1,\ldots,N-1\}^d$ such that $\bsh \cdot \bsg \equiv 0 \pmod{N}$. For every such $\bsg$ we have $\kappa(\XX^{\bot}(\bsg,N))\le \kappa$. Therefore
\begin{equation*}
|\{\bsg \in \{0,1,\ldots, N-1\}^d \ : \ \kappa(\XX^{\bot}(\bsg,N)) \le \kappa \}| 
\le (N-1)^{d-1} 2^d \binom{\kappa+d}{d}.
\end{equation*}
Note that the total number of possible generators 
$\bsg \in \{1,\ldots, N-1\}^d$ is $(N-1)^d$. Thus, if  
\begin{equation}\label{eq_exist_n_1}
(N-1)^{d-1} 2^d \binom{\kappa+d}{d} < \frac{1-\delta}{3} (N-1)^d,
\end{equation}
then the number of $\bsg \in \{1,\ldots, N-1\}^d$ 
such that $\kappa(\XX^{\bot}(\bsg,N)) > \kappa$ is at least $((2+\delta)/3)(N-1)^ d$.
We estimate
\begin{equation}\label{est1}
2^d \binom{\kappa + d}{d} \le \frac{2^d (\kappa+d)^d}{d!}.
\end{equation}
Thus \eqref{eq_exist_n_1} is satisfied 
if $2^d (\kappa+d)^d /d! < (1-\delta) (N-1)/3$, i.e., for $\kappa = \lceil 2^{-1} (d!\,(1-\delta) (N-1)/3)^{1/d} \rceil -d -1$. Hence, for at least $((2+\delta)/3)(N-1)^d$ many $\bsg \in \{1,\ldots, N-1\}^d$ we have 
\begin{equation}\label{est:kappa1}
\kappa(\XX^{\bot}(\bsg,N)) \ge B_{d,\delta}' N^{1/d} 
\end{equation}
for some $B_{d,\delta}'>0$, depending only on $d$ and $\delta$.

Now we consider $\kappa(\XX(\bsg,N))$. The elements in $\XX(\bsg,N)$ are of the form $(1/N) \bsh$ with $\bsh \in \ZZ^d$. Hence $$\kappa(\XX(\bsg,N))=\frac{1}{N} \min_{(1/N) \bsh \in \XX(\bsg,N)\setminus\{\bszero\}} \|\bsh\|_1.$$

For a given $\bsh \in A_\kappa \setminus \{\bszero\}$ with $\kappa<N$, we count the number of $\bsg \in \{1,\ldots,N-1\}^d$ such that $(1/N) \bsh \in \XX(\bsg,N)$. This is equal to the number of elements in the set $$\{\bsg \in \{1,\ldots,N-1\}^d \ : \ \exists \ell \in \{0,1,\ldots,N-1\} \mbox{ such that } \bsh \equiv \ell \bsg \pmod{N}\}.$$ For $\ell \not=0$ we have $$\bsh \equiv \ell \bsg \pmod{N} \ \mbox{ if and only if } \ \bsg \equiv \ell^{-1} \bsh \pmod{N},$$ where $\ell^{-1}$ is the inverse of $\ell$ modulo $N$ (remember that $N$ is a prime). Hence, for every $\ell \not=0$, there is exactly one solution $\bsg$ if $h_j\not\equiv 0 \pmod N$ for all $j$; otherwise, no such $\bsg$ exists. For the case $\ell=0$, there is no solution, since the condition reduces to $\bsh\equiv \bszero \pmod N$, which contradicts the assumption $\bsh \in A_\kappa \setminus \{\bszero\}$ with $\kappa<N$. This means that the number of $\bsg \in \{1,\ldots,N-1\}^d$ such that $(1/N) \bsh \in \XX(\bsg,N)$ is at most $N-1$. 

Therefore, for any $1\le \kappa<N$,
\begin{equation*}
\left|\left\{\bsg \in \{1,\ldots, N-1\}^d \ : \ \kappa(\XX(\bsg,N)) \le \frac{\kappa}{N} \right\}\right| 
\le (N-1) 2^d \binom{\kappa+d}{d}.
\end{equation*}

Like above, if  
\begin{equation}\label{eq_exist_n_2}
(N-1) 2^d \binom{\kappa+d}{d} < \frac{1-\delta}{3} (N-1)^d,
\end{equation}
then the number of $\bsg \in \{1,\ldots, N-1\}^d$ 
such that $\kappa(\XX(\bsg,N)) > \kappa/N$ is at least $((2+\delta)/3)(N-1)^d$.
According to \eqref{est1}, condition~\eqref{eq_exist_n_2} is satisfied 
if $2^d (\kappa+d)^d /d! < (1-\delta) (N-1)^{d-1}/3$, i.e., for $$\kappa = \min(\lceil 2^{-1} (d!\,(1-\delta) (N-1)^{d-1}/3)^{1/d} \rceil -d -1, N-1).$$ Hence, for at least $((2+\delta)/3)(N-1)^ d$ many $\bsg \in \{1,\ldots, N-1\}^d$ we have 
\begin{equation}\label{est:kappa2}
\kappa(\XX^{\bot}(\bsg,N)) \ge B_{d,\delta}'' \frac{N^{(d-1)/d}}{N} = \frac{B_{d,\delta}''}{N^{1/d}}
\end{equation}
for some $B_{d,\delta}''>0$, depending only on $d$ and $\delta$.

Now we combine the three parts. We have:
\begin{itemize}
\item for the set $A$ of all $\bsg \in \{1,\ldots, N-1\}^d$ which satisfy \eqref{est:disc} we have $|A| \ge ((2+\delta)/3)(N-1)^d$,
\item for the set $B$ of all $\bsg \in \{1,\ldots, N-1\}^d$ which satisfy \eqref{est:kappa1} we have $|B| \ge ((2+\delta)/3)(N-1)^d$,
\item for the set $C$ of all $\bsg \in \{1,\ldots, N-1\}^d$ which satisfy \eqref{est:kappa2} we have $|C| \ge ((2+\delta)/3)(N-1)^d$.
\end{itemize}
Now $A \cap B \cap C$ consists of all $\bsg\in \{1,\ldots, N-1\}^d$, for which \eqref{est:disc}, \eqref{est:kappa1} and \eqref{est:kappa2} hold simultaneously and we have 
\begin{align*}
|A\cap B \cap C| = & (N-1)^d - |A^c \cup B^c \cup C^c| \\
\ge &  (N-1)^d - |A^c| -|B^c| - |C^c|\\
= & (N-1)^d -3 (N-1)^d + |A|+|B|+|C|\\
\ge & \delta (N-1)^d.
\end{align*}
Obviously, the number of elements is an integer and hence $|A\cap B \cap C| \ge \lceil \delta (N-1)^d \rceil$.

Finally, \eqref{est:kappa1} and \eqref{est:kappa2} lead to $$\frac{1}{\kappa(\XX(\bsg,N))\kappa(\XX^{\bot}(\bsg,N))} \le \frac{1}{B_{d,\delta}' B_{d,\delta}''}.$$ The estimate on $\rho_p(P(\bsg,N))$ follows with \eqref{est:rho:kappa}.
\end{proof}

Let $\mathbb{P}$ be the set of prime numbers. Since the sequence of prime numbers $(p_n)_{n \ge 1}$, where $p_n$ is the $n$-th prime number, satisfies $p_{n+1} < 2p_n$ by Bertrand's postulate, we obtain that there exists a sequence of rank-1 lattice point sets $(P_N)_{N \in \mathbb{P}}$ which are quasi-uniform and which have low star-discrepancy.

\begin{remark} 
    The result of Theorem~\ref{thm:exlatd} can be generalized to any lower-dimensional projection. That is, for given $d \in \NN$, $p \in [1,\infty]$, and $\delta \in (0,1)$  for any $s \in \{1,\ldots,d\}$, there exist quantities $B_{s,p, \delta},C_ {s,p, \delta}>0$ such that, for fixed non-empty $\uu \subseteq \{1,\ldots,d\}$ and for every prime number $N$ the bounds
    \begin{align}\label{eq:bound_lower_proj}
    \rho_p(P_\uu(\bsg,N)) \le B_{|\uu|,p,\delta}\qquad \mbox{and} \qquad D_N^{\ast}(P_\uu(\bsg,N)) \le C_{|\uu|,p,\delta} \frac{(\log N)^{|\uu|}}{N},
    \end{align}
    hold for at least $\lceil \delta N^d\rceil$ vectors $\bsg \in \{1,\ldots,N-1\}^d$, where $P_\uu(\bsg,N)\subseteq [0,1)^{|\uu|}$ denotes the projection of the point set $P(\bsg,N)$ onto the coordinates indexed by $\uu$.

    This implies that for any $\delta \in (1-2^{-d},1)$, there exists at least one vector $\bsg \in \{1,\ldots,N-1\}^d$ such that \eqref{eq:bound_lower_proj} holds simultaneously for all $\uu \subseteq \{1,\ldots,d\}$, i.e., for all lower-dimensional projections of the lattice point set. In this case, the quantities $B_{|\uu|,p, \delta}$ and $C_ {|\uu|,p, \delta}$ depend both on the dimension $d$ through their dependence on $\delta$. Thus there exists a sequence of lattice point sets $(P(\bsg, N))_{N \in \mathbb{P}}$, where $\mathbb{P}$ is the set of prime numbers, for which all projections onto the coordinates in $\uu \subseteq \{1, \ldots, d\}$ are quasi-uniform and for which all projections have small discrepancy.
\end{remark}

\subsection{Kronecker sequences}\label{secc:nalpha}

So-called Kronecker sequences (sometimes also referred to as $(n \boldsymbol{\alpha})$-sequences) are sequences of the form $(\{n \boldsymbol{\alpha}\})_{n \ge 1}$, where $\boldsymbol{\alpha}=(\alpha_1,\ldots,\alpha_d) \in \RR^d$ (sometimes in literature the sequence starts with $\boldsymbol{0}=0 \boldsymbol{\alpha}$). These are important and well-studied objects in uniform distribution- and discrepancy-theory (see, e.g, \cite[Section~1.4]{DT97} and the references therein). It follows easily from Weyl's criterion that an $(n \boldsymbol{\alpha})$-sequence is uniformly distributed modulo 1 if and only if $1,\alpha_1,\ldots,\alpha_d$ are linearly independent over the rationals. Usually, results on the discrepancy of $(n \boldsymbol{\alpha})$-sequences depend on Diophantine properties of the real vector $\boldsymbol{\alpha}$. Regarding this topic a multitude of results are available in the literature, see, e.g., \cite{DT97,KN74,N73} and the references therein.

In this section we characterize those $\boldsymbol{\alpha}$, for which the $(n \boldsymbol{\alpha})$-sequences have the property of being quasi-uniform. The one-dimensional case has already been discussed in \cite{G24b}.

For $x \in \mathbb{R}$ let $\langle x \rangle = \min_{k \in \mathbb{Z}} |x-k|$ denote the distance of $x$ to the nearest integer. For a vector $\boldsymbol{x} = (x_1, x_2, \ldots, x_d) \in \mathbb{R}^d$ let $\langle \boldsymbol{x} \rangle = \max_{1 \le j \le d} \langle x_j \rangle$.

\begin{theorem}\label{thm:nalpha}
For $\boldsymbol{\alpha} \in \mathbb{R}^d$ the $(n \boldsymbol{\alpha})$-sequence $(\{n \boldsymbol{\alpha}\})_{n \ge 1}$ is quasi-uniform, if and only if there exists a constant $c > 0$ such that
\begin{equation}\label{nalpha_exists}
\langle n \boldsymbol{\alpha} \rangle \ge \frac{c}{n^{1/d}} \qquad \mbox{for all } n \in \mathbb{N}.
\end{equation}
\end{theorem}

\begin{remark}
Property \eqref{nalpha_exists} characterizes so-called badly approximable vectors $\boldsymbol{\alpha}$, which is an important concept in Diophantine approximation theory.   
\end{remark}

For the proof of Theorem~\ref{thm:nalpha} we will need the following auxiliary results.

\begin{lemma}\label{le:3.11} 
Let $\boldsymbol{\alpha}$ be such that \eqref{nalpha_exists} holds. Let $Q_1 < Q_2 < \ldots $ be the best simultaneous approximation denominators for $\boldsymbol{\alpha}$ with respect to maximum norm. That means  $Q_1 := 1$ and
$$Q_i := \min \{ Q > Q_{i-1} \ : \  \langle Q \boldsymbol{\alpha} \rangle  < \langle Q_{i-1} \boldsymbol{\alpha} \rangle \} \qquad \mbox{for $i \in \mathbb{N}\setminus \{1\}$.}$$  
Then $$Q_i \le c_1  Q_{i-1}$$ for all $i$, where $c_1$ is an absolute constant depending only on $d$ and on $c$ from \eqref{nalpha_exists}.
\end{lemma}

\begin{proof}
By the $d$-dimensional version of Dirichlet's theorem in Diophantine approximation (see, for example, \cite[Chapter~I.5]{C57}) for every  $N \in \mathbb{N}$ there exists an $n \le N$ such that  $\langle n \boldsymbol{\alpha} \rangle  <  N^{-1/d}$.

 Applying this to  $N := Q_i-1$ by the definition of the $Q_i$ and by (20) we obtain
 $$\frac{c}{Q_{i-1}^{1/d}} \le \langle Q_{i-1} \boldsymbol{\alpha} \rangle  <  \frac{1}{(Q_i-1)^{1/d}}.  $$
 Hence $$\left( \frac{Q_i-1}{Q_{i-1}}\right)^{1/d}  \le \frac{1}{c} $$ and consequently $$ \frac{Q_i}{Q_{i-1}} \le 2 \, \left( \frac{1}{c} \right)^d$$
 and the result follows.
\end{proof}

Let now $Q$ be any of the best approximation denominators $Q_i$. Then by Dirichlet we have  $\langle Q \boldsymbol{\alpha} \rangle  <  \frac{1}{Q^{1/d}}$. Hence
$$ \alpha_j = \frac{P_j}{Q} + \frac{\delta_j}{Q^{1+1/d}}$$
for all $j \in \{1, 2, \ldots , d\}$, with $P_1, P_2, \ldots , P_d$ integers with $$\gcd(Q, P_1, \ldots , P_d) = 1$$ and $ | \delta_j | \le 1$ for all $j$.
 
Let $\XX := \XX((P_1,\ldots,P_d),Q)$ be the lattice in $\mathbb{R}^d$ generated by a matrix $T$ satisfying $\det(T)=1/Q$. Note that when $\gcd(P_1,Q)=1$, $T$ takes the explicit form \eqref{eq:r1lat} with $N$ and $\bsg$ replaced by $Q$ and $(P_1,\ldots,P_d)$, respectively.

Let $\lambda_1 \le \lambda_2 \le \cdots \le \lambda_d$ denote the successive minima of $\XX$ with respect to the $\ell_\infty$ norm.

\begin{lemma}\label{le:3.12}  
For $\boldsymbol{\alpha}$ satisfying \eqref{nalpha_exists}, there exists a constant $c_2 := c_2(d,c) > 0$, depending only on the dimension $d$ and on $c$ from \eqref{nalpha_exists}, such that $$\lambda_1 \ge \frac{c_2}{Q^{1/d}}.$$ 
\end{lemma}

\begin{proof}
Assume that $\lambda_1  <  \varepsilon/Q^{1/d}$. We will show that this leads to a contradiction if $\varepsilon$ is chosen sufficiently small, depending only on $d$ and the constant $c$ from \eqref{nalpha_exists}. 
By the definition of $\lambda_1$, this means that there exists a nonzero lattice vector of norm less than $\varepsilon/Q^{1/d}$. In particular, this implies that for some  $0 < \ell < Q$ we have $$\max_j \left\langle \ell \, \frac{P_j}{Q} \right\rangle \le \frac{\varepsilon}{Q^{1/d}}.$$ Furthermore, for $j \in \{ 1, 2, \ldots , d\}$, we have $$\ell \, \alpha_j = \ell \, \frac{P_j}{Q} + \frac{\delta_j\,\ell}{Q^{1+1/d}}.$$

For such an $\ell$, consider now the values $u \ell \pmod{Q}$ for $u = 1, 2, \ldots, U$, where we will choose $U$ later depending only on $d$ and $c$. Let $v$ denote the minimal distance of these values to $0$ on the circle $\RR/Q\ZZ$ (i.e., the circular distance). Then, Dirichlet's approximation theorem ensures that this minimal distance, attained at some index $a\in \{1,\ldots,U\}$, satisfies
\[ 1\le v\le \frac{Q}{U+1}\le \frac{Q}{U}.\]
Here, $v$ is explicitly given by $v = \min(a\ell \bmod Q, \, Q - (a\ell \bmod Q))$.

We have $$\max_j \left\langle v \, \frac{P_j}{Q} \right\rangle = \max_j \left\langle a \ell \, \frac{P_j}{Q} \right\rangle \le a \max_j \left\langle  \ell \, \frac{P_j}{Q} \right\rangle\le  \frac{a \, \varepsilon}{Q^{1/d}} \le \frac{U \, \varepsilon }{Q^{1/d}}.$$
Furthermore,
$$ \langle v \, \alpha_j \rangle = \left\langle v \, \frac{P_j}{Q} + \frac{\delta_j \, v}{Q^{1+1/d}} \right\rangle$$
and hence
$$\max_j\, \langle v \, \alpha_j \rangle \le \max_j \left\langle v \, \frac{P_j}{Q}\right\rangle + \frac{v}{Q^{1+1/d}} \le \frac{U\, \varepsilon}{Q^{1/d}}+\frac{1}{U \, Q^{1/d}}. 
 $$
On the other hand, by \eqref{nalpha_exists} we have 
$$\max_j\, \langle v \, \alpha_j \rangle \ge \frac{c}{v^{1/d}} \ge \frac{c \, U^{1/d}}{Q^{1/d}}.$$
So we must have
$$\frac{c \, U^{1/d}}{Q^{1/d}} \le \frac{U\, \varepsilon}{Q^{1/d}}+\frac{1}{U \, Q^{1/d}}, $$
i.e.,
$$ c\, U^{1/d} \le U \, \varepsilon + \frac{1}{U}.$$
But this is not satisfied if we choose $U$ so large that $c \, U^{1/d} - \frac{1}{U} > 1$ and $\varepsilon < \frac{1}{U}$. This gives the contradiction. 
\end{proof}

Now we can proceed with the proof of Theorem~\ref{thm:nalpha}:

\begin{proof}[Proof of Theorem~\ref{thm:nalpha}]
Combining Lemmata~\ref{lem_seq_weak} and \ref{le:3.11}, we see that it suffices to prove the boundedness of the mesh ratio for the subsequence $Q_2<Q_3<\cdots$, since this implies the boundedness for the first $N$ points for all $N\ge 2$. In what follows, we choose any $Q_i$ and denote it by $Q$.

First, we show that \eqref{nalpha_exists} is sufficient for quasi-uniformity. We use the notations of the Lemmata~\ref{le:3.11} and \ref{le:3.12}.

We first give an appropriate upper bound on the covering radius.
From Lemma~\ref{le:3.12} and discussion on Lemmata~\ref{lem:cov-by-minima}, \ref{lem:minima_meshratio} and \ref{le3.3},
there exists a constant $c_3>0$ such that the covering radius on the rank-$1$ lattice $\mathbb{X} \cap [0,1)^d$ satisfies
\begin{equation}\label{eq:sep-approxlat}
h_\infty(\mathbb{X} \cap [0,1)^d) \le c_3 Q^{-1/d}.    
\end{equation}

Choose an arbitrary $\boldsymbol{x} \in [0,1)^d$.
It follows from \eqref{eq:sep-approxlat} that there exists $\boldsymbol{z} \in \mathbb{X} \cap [0,1)^d$ with $\|\boldsymbol{x}-\boldsymbol{z}\|_\infty \le c_3/Q^{1/d}$.
Let $\boldsymbol{z} = ( \{ m \, \frac{P_1}{Q} \}, \ldots , \{ m \, \frac{P_d}{Q} \} )$  with $0 \le m < Q$. 
Assume first, that $$\boldsymbol{x} \in C_Q^{(0)}:=\left[\frac{c_3+1}{Q^{1/d}},1-\frac{c_3+1}{Q^{1/d}}\right)^d.$$
Then we have 
\begin{equation}\label{eq:z-cube}
\boldsymbol{z}\in C_Q^{(1)}:=\left[\frac{1}{Q^{1/d}},1-\frac{1}{Q^{1/d}}\right)^d.    
\end{equation}
Together with
\[
\max_{1\le j\le d}\langle z_j-\{m\alpha_j\}\rangle
\le \max_j \frac{Q|\delta_j|}{Q^{1+1/d}}
\le \frac{1}{Q^{1/d}},
\]
we must have
$|z_j-\{m\alpha_j\}| = \langle z_j-\{m\alpha_j\} \rangle$. 
Thus we have
\[\|\boldsymbol{z} - \{m\boldsymbol{\alpha}\}\|_\infty = \langle \boldsymbol{z}-\{m \boldsymbol{\alpha}  \}\rangle \le \frac{1}{Q^{1/d}}.\]
Hence, for such $\boldsymbol{x}$ we have 
$$ \| \boldsymbol{x} - \{ m \, \boldsymbol{\alpha} \} \|_{\infty}  \le  \| \boldsymbol{x} - \boldsymbol{z}\|_{\infty}  +  \| \boldsymbol{z} - \{ m \, \boldsymbol{\alpha} \} \|_{\infty}  \le  \frac{c_3}{Q^{1/d}} + \frac{1}{Q^{1/d}} \le \frac{\tilde{c}}{Q^{1/d}}.  $$
So the desired bound for the covering radius follows.

If $\boldsymbol{x} \in [0,1)^d \setminus C_Q^{(0)}$, then choose any $\widetilde{\boldsymbol{x}} \in C_Q^{(0)}$ such that $\|\boldsymbol{x}-\widetilde{\boldsymbol{x}}\|_{\infty} \le c_3/Q^{1/d}$. Such a $\widetilde{\boldsymbol{x}}$ certainly exists. For $\widetilde{\boldsymbol{x}}$ by the above reasoning we again find a suitable $\{m \boldsymbol{\alpha}\}$, and $\|\boldsymbol{x}-\{m \boldsymbol{\alpha}\}\|_\infty \le \|\boldsymbol{x}-\widetilde{\boldsymbol{x}}\|_\infty + \|\widetilde{\boldsymbol{x}}-\{m \boldsymbol{\alpha}\}\|_\infty$ also in this case leads to the desired bound for the covering radius of order $Q^{-1/d}$.

Now consider separation. Let $N \ge M \ge 1$. Then \eqref{nalpha_exists} implies
\begin{equation*}
\langle \{ N \boldsymbol{\alpha} \} - \{ M \boldsymbol{\alpha} \} \rangle = \langle (N - M) \boldsymbol{\alpha} \rangle \ge \frac{c}{(N- M)^{1/d}} \ge \frac{c}{N^{1/d}}.
\end{equation*}
Thus we obtain
\begin{equation*}
\| \{N \boldsymbol{\alpha}\} - \{M \boldsymbol{\alpha}\} \|_\infty \ge \frac{c}{N^{1/d}}.
\end{equation*}
Summing up, \eqref{nalpha_exists} implies quasi-uniformity.

In order to prove the converse direction, assume that \eqref{nalpha_exists} does not hold. Then for every $\varepsilon > 0$ there exists an $n_0$ with
\begin{equation*}
\langle n_0 \boldsymbol{\alpha} \rangle \le \frac{\varepsilon}{n_0^{1/d}}.
\end{equation*}
Consider two points $\boldsymbol{v} = (v_1, \dots, v_d) := \{n_0 \boldsymbol{\alpha} \} $ and $\boldsymbol{u} = (u_1, \dots, u_d) := \{2 n_0 \boldsymbol{\alpha} \}$. Then for any $j \in \{1,\ldots,d\}$
we have $u_j = 2v_j$ if $0 \le v_j < 1/2$ and $u_j = 2v_j-1$ otherwise,
and thus we have
\[
\langle u_j  - v_j \rangle = \langle v_j \rangle
\quad \text{and} \quad
|u_j  - v_j| =
\left\{ 
\begin{array}{ll}
|v_j| & \text{if $0 \le v_j < 1/2$}\\
|1-v_j| & \text{otherwise}
\end{array}
\right\} = \langle v_j \rangle.
\]
Thus we have
\begin{equation}\label{cond:pr:alpha}
\|\boldsymbol{u}-\boldsymbol{v}\|_\infty
= \langle \boldsymbol{v} \rangle
\le \frac{\varepsilon}{n_0^{1/d}}.    
\end{equation}
If the $(n \boldsymbol{\alpha})$-sequence is quasi-uniform, then there exists a $c_d>0$ such that
\[
\|\boldsymbol{u}-\boldsymbol{v}\|_\infty
= \| \{ 2 n_0 \boldsymbol{\alpha} \} - \{ n_0 \boldsymbol{\alpha} \} \|_\infty \ge \frac{c_d}{(2n_0)^{1/d}}.\]
However, in \eqref{cond:pr:alpha} the $\varepsilon$ can be chosen arbitrarily close to 0, and this yields a contradiction. Thus, the $(n \boldsymbol{\alpha})$-sequence cannot be quasi-uniform and the result follows.

\end{proof}

From \cite[Chapter 1; Proof of Theorem VIII]{C57} it follows that vectors $\boldsymbol{\alpha}$ exist which satisfy Condition~\eqref{nalpha_exists}. In particular, see \cite[Chapter 1; Eq. (9)]{C57}. Theorem~\ref{thm:nalpha} generalizes the result shown in \cite{G24b} for one-dimensional Kronecker sequences. Also there, it was shown that the sufficient and necessary condition for the sequence to be quasi-uniform is that the real number $\alpha$ is badly approximable. This condition is different from the condition for the sequence to be of low discrepancy. In the multi-dimensional case, \cite[Theorem~1]{L88} shows that the isotropic discrepancy $J_N(\Scal)$ of $(n \boldsymbol{\alpha})$-sequences $\Scal$ satisfying \eqref{nalpha_exists} is bounded by $J_N(\Scal) \ll N^{-1/d}$ for $d \in \mathbb{N}\setminus \{1\}$ and the same upper bound applies to the star-discrepancy of $\Scal$. In particular, Condition~\eqref{nalpha_exists} implies that the star-discrepancy $D_N^{\ast}$ of an $(n \boldsymbol{\alpha})$-sequence tends to 0 as $N$ goes to infinity and hence we obtain the following result:

\begin{corollary}\label{cor:qud:ud}
If an $(n \boldsymbol{\alpha})$-sequence is quasi-uniform, then it is uniformly distributed.    
\end{corollary}

The result of Corollary~\ref{cor:qud:ud} is a specific feature of $(n \boldsymbol{\alpha})$-sequences which must not necessarily hold for general sequences. A counterexample is the Sobol' sequence, which is uniformly distributed (its discrepancy is of order of magnitude $(\log N)^d/N$) but which is not quasi-uniform in dimension $2$, as shown in \cite{G24a}.

The bounds on the star-discrepancy mentioned above are not optimal. But we can provide explicit examples of $(n \boldsymbol{\alpha})$-sequences which are quasi-uniform and have at the same time a very low (in fact, almost best possible) star-discrepancy.

We will use the following discrepancy bound for Kronecker-sequences (see, e.g, \cite[Theorem~6.1]{N73}, \cite[Section~4.8]{DHP15} or \cite[Chapter~2, Section~3, Exercises~3.17]{KN74}): Let $\alpha_1,\ldots,\alpha_d$ be irrational numbers. If for all $\varepsilon>0$ there is a $c(\varepsilon)>0$ such that 
\begin{equation}\label{cond:alpha}
(\overline{n}_1 \cdots \overline{n}_d)^{1+\varepsilon} \langle n_1 \alpha_1+\cdots+n_d \alpha_d \rangle > c(\varepsilon) \quad \mbox{for all } (n_1,\ldots,n_d) \in \mathbb{Z}^d\setminus \{\boldsymbol{0}\},    
\end{equation}
where $\overline{n}:=\max(1,|n|)$, then for the discrepancy of the corresponding $(n \boldsymbol{\alpha})$-sequence $\mathcal{S}$ we have $$D_N^{\ast}(\mathcal{S})\ll \frac{1}{N^{1-\delta}} \quad \mbox{for all } \delta>0.$$

It was shown by Schmidt~\cite{S70} that \eqref{cond:alpha} holds whenever $\alpha_1,\ldots,\alpha_d$ are algebraic and such that $1,\alpha_1,\ldots,\alpha_d$ are linearly independent over $\mathbb{Q}$ (this generalizes the famous Thue-Siegel-Roth Theorem).

Let now $1,\alpha_1,\ldots,\alpha_d$ be linear independent elements (hence a basis) of an algebraic extension of $\mathbb{Q}$ of degree $d+1$. Then, according to \cite[Chapter~V, Theorem~3]{C57}, $\alpha_1,\ldots,\alpha_d$ also satisfy property \eqref{nalpha_exists}. Hence the corresponding $(n \boldsymbol{\alpha})$-sequence $\mathcal{S}$ in this case is quasi-uniform and its star-discrepancy satisfies $$D_N^{\ast}(\mathcal{S}) \ll \frac{1}{N^{1-\delta}} \quad \mbox{for all } \delta>0.$$ 

As a concrete example we can choose for instance 
\begin{equation*}
\boldsymbol{\alpha} = ( 2^{1/(d+1)}, 2^{2/(d+1)}, \ldots, 2^{d/(d+1)}) \in \mathbb{R}^d. 
\end{equation*}
Then the corresponding $(n \boldsymbol{\alpha})$-sequence $\mathcal{S}$ is quasi-uniform and its star-discrepancy is of order $N^{-1+\delta}$ for any $\delta > 0$.

\appendix
\section{Results and Proofs for quasi-uniformity}
\subsection{Radii bounds for the unit cube}\label{sec:unitcube_bound}

Here we give bounds on covering/separation radii via geometric interpretation.
Let us denote the closed $\ell_p$ ball of center $\bsx$ and radius $r$ by $B_{p}(\bsx,r)$.
Since the covering radius is the smallest value of $r$ such that the union of the balls covers the entire domain $[0,1]^d$, we have
\[ \bigcup_{\bsx\in P}B_p(\bsx,h_p(P))\supseteq [0,1]^d. \]
With $N$ being the number of points in $P$, it follows that
\begin{align*}        
1 = \mathrm{vol}([0,1]^d)&\leq \mathrm{vol}\left( \bigcup_{\bsx\in P}B_p(\bsx,h_p(P))\right)\\
&\leq \sum_{\bsx\in P}\mathrm{vol}\left( B_p(\bsx,h_p(P))\right)=N \frac{(2\Gamma(1+1/p))^d}{\Gamma(1+d/p)}(h_p(P))^d,
\end{align*}
where $\Gamma$ denotes the gamma function. 
Thus the covering radius is bounded below by
\begin{equation}
h_p(P)\ge \frac{1}{N^{1/d}}\ \frac{(\Gamma(1+d/p))^{1/d}}{2\Gamma(1+1/p)}.
\end{equation}

Similarly, it follows from the interpretation of the separation radius that a slightly enlarged domain $\left[-q_p(P),1+q_p(P) \right]^d$ must contain all of the $N$ open $\ell_p$ balls with radius $q_p(P)$, which do not intersect with each other. Hence, we have
\begin{align*}
\left(1+2q_p(P)\right)^d & = \mathrm{vol}\left( \left[-q_p(P),1+q_p(P) \right]^d\right) \geq \mathrm{vol}\left( \bigcup_{\bsx\in P}B_p(\bsx,q_p(P))\right)\\
&= \sum_{\bsx\in P}\mathrm{vol}\left( B_p(\bsx,q_p(P))\right)=N \frac{(2\Gamma(1+1/p))^d}{\Gamma(1+d/p)}(q_p(P))^d.
\end{align*} 
Thus the separation radius is bounded above by
\begin{equation}    
q_p(P)\le \frac{(\Gamma(1+d/p))^{1/d}}{2N^{1/d}\Gamma(1+1/p)-2(\Gamma(1+d/p))^{1/d}}.\end{equation}

\subsection{Proof of Lemma~\ref{lem_seq_weak}}\label{sec:app_proof1}
Since the mesh ratio $\rho_p(P_{i_k})$ is bounded, the covering radius and separation radius must both be of the optimal order. Hence there are constants $C_p, C'_p > 0$ such that
\begin{align*}
h_p(P_{i_k}) \le C_p i_k^{-1/d} \quad \mbox{and} \quad q_p(P_{i_k}) \ge C'_p i_k^{-1/d} \quad \mbox{for all } k \in \mathbb{N}.
\end{align*}

Let $m \in \mathbb{N}$ such that $i_{k} \le m < i_{k+1}$. Then $i_k \le m < c i_k$. Let $P_m = \{\bsx_0, \bsx_1, \ldots, \bsx_{m-1}\}$, then $P_{i_k} \subseteq P_m \subset P_{i_{k+1}}$. Thus the covering radius satisfies
\begin{equation*}
h_p(P_m) \le h_p(P_{i_k}) \le C i_{k}^{-1/d} \le C_p (m/c)^{-1/d}.
\end{equation*}
and the separation radius satisfies
\begin{equation*}
q_p(P_m) \ge q_p(P_{i_{k+1}}) \ge C'_p i_{k+1}^{-1/d} \ge C'_p (c i_{k})^{-1/d} \ge C'_p (c m)^{-1/d}.
\end{equation*}
Thus the mesh ratio of $P_m$ satisfies
\begin{equation*}
\rho_p(P_m) \le \frac{C_p}{C'_p} c^{2/d}.
\end{equation*}

Note that $\rho_p(P_{i_1}) \le C''_p$, hence the points $\bsx_0, \bsx_1, \ldots, \bsx_{i_1-1}$ are all distinct and so the separation radius of $P_1, P_2, \ldots, P_{i_1-1}$ is bounded away from $0$. Since the volume of the unit cube is finite, also the covering radius is finite. Hence the mesh ratios of $P_1, P_2, \ldots, P_{i_1-1}$ are all finite.

Thus for any $m \in \mathbb{N}$ we have
\begin{equation*}
\rho_p(P_m) \le \max \{C_p c^{2/d} / C'_p, \rho_p(P_1), \rho_p(P_2), \ldots, \rho_p(P_{i_1-1})\} < \infty.
\end{equation*}

Hence the mesh ratio is bounded independently of $m$ and therefore $\mathcal{S}$ is a quasi-uniform sequence. The remaining parts can be shown in a similar way.

\subsection{Proof of Lemma~\ref{lem_seq_weak_inv}}\label{sec:app_proof2}
Let $\Scal=\{x_0,x_1,\ldots\}$ be the van der Corput sequence in base $2$, which is quasi-uniform.
Take $k,\ell,u,v \in \NN$ such that
$2^{u-1} < k \le 2^u \le 2^v < \ell \le 2^{v+1}.$
Consider $\Scal'$, obtained from $\Scal$ by swapping the two points $x_{k}$ and $x_{2^v}$,
and let $P'_k$ 
be the first $k$ points from $\Scal'$.
Then $P_k = P'_k$ and $P_\ell = P'_\ell$, whereas
\[
q_p(P'_{k+1})
 = \frac{|x_0 - x_{2^v}|}{2}
 = 2^{-v-1},
 \qquad
h_p(P'_{k+1})
 \ge |1 - (2^u - 1)2^{-u}|
 = 2^{-u},
\]
and therefore
\[
\rho_p(P'_{k+1})
 = \frac{h_p(P'_{k+1})}{q_p(P'_{k+1})}
 \ge 2^{v-u+1}
 \ge \frac{\ell}{2k}.
\]

Thus, for any increasing indices $1 \le i_1 < i_2 < i_3 < \cdots$ with 
$\sup_{k\in\NN} i_{k+1}/i_k = \infty$,
repeating the above kind of swap on the van der Corput sequence $\Scal$ produces a sequence $\Scal'$ for which $\rho(P'_{i_k}) = \rho(P_{i_k})$ remains bounded, while
\[
\sup_{k \in \NN}\rho(P'_{i_k+1}) \ge \sup_{k \in \NN}\frac{i_{k+1}}{2i_k} = \infty.
\]
This shows that $\Scal'$ is not quasi-uniform.

\bibliographystyle{plain}
\bibliography{ref.bib}

@article {ABD12,
    AUTHOR = {Aistleitner, C. and Brauchart, J. S. and Dick, J.},
     TITLE = {Point sets on the sphere {$\Bbb{S}^2$} with small spherical
              cap discrepancy},
   JOURNAL = {Discrete Comput. Geom.},
  FJOURNAL = {Discrete \& Computational Geometry. An International Journal
              of Mathematics and Computer Science},
    VOLUME = {48},
      YEAR = {2012},
    NUMBER = {4},
     PAGES = {990--1024},
}

@article {Aron50,
    AUTHOR = {Aronszajn, N.},
     TITLE = {Theory of reproducing kernels},
   JOURNAL = {Trans. Amer. Math. Soc.},
  FJOURNAL = {Transactions of the American Mathematical Society},
    VOLUME = {68},
      YEAR = {1950},
     PAGES = {337--404},
      ISSN = {0002-9947,1088-6850},
   MRCLASS = {46.0X},
  MRNUMBER = {51437},
MRREVIEWER = {T.\ H.\ Hildebrandt},
       DOI = {10.2307/1990404},
       URL = {https://doi.org/10.2307/1990404},
}

@Article{beck,
  author = {Beck, J.},
  year = {1988},
  title = {On the discrepancy of convex plane sets},
  journal = {Monatsh. Math.},
  volume = {105},
  number = {2},
  pages = {91--106}
}

@article {blv08,
    AUTHOR = {Bilyk, D. and Lacey, M.~T. and Vagharshakyan, A.},
     TITLE = {On the small ball inequality in all dimensions},
   JOURNAL = {J. Funct. Anal.},
    VOLUME = {254},
    NUMBER = {9},
      YEAR = {2008},
     PAGES = {2470--2502}
}

@book {C57,
    AUTHOR = {Cassels, J. W. S.},
     TITLE = {An {I}ntroduction to {D}iophantine {A}pproximation},
    SERIES = {Cambridge Tracts in Mathematics and Mathematical Physics},
    VOLUME = {No. 45},
 PUBLISHER = {Cambridge University Press, New York},
      YEAR = {1957},
     PAGES = {x+166},
   MRCLASS = {10.3X},
  MRNUMBER = {87708},
MRREVIEWER = {H.\ Davenport},
}

@article {DHP15,
    AUTHOR = {Dick, J. and Hinrichs, A. and Pillichshammer, F.},
     TITLE = {Proof techniques in quasi--{M}onte {C}arlo theory},
   JOURNAL = {J. Complexity},
  FJOURNAL = {Journal of Complexity},
    VOLUME = {31},
      YEAR = {2015},
    NUMBER = {3},
     PAGES = {327--371}
}

@book {DKP22,
    AUTHOR = {Dick, J. and Kritzer, P. and Pillichshammer, F.},
     TITLE = {Lattice Rules---Numerical Integration, Approximation, and
              Discrepancy},
    SERIES = {Springer Series in Computational Mathematics},
    VOLUME = {58},
 PUBLISHER = {Springer, Cham},
      YEAR = {2022},
     PAGES = {xvi+580},
}

@book {DP10,
    AUTHOR = {Dick, J. and Pillichshammer, F.},
     TITLE = {Digital Nets and Sequences---Discrepancy Theory and Quasi-Monte Carlo Integration},
     PUBLISHER = {Cambridge University Press, Cambridge},
      YEAR = {2010},
     PAGES = {xviii+600},
}

@book {DT97,
    AUTHOR = {Drmota, M. and Tichy, R.F.},
     TITLE = {Sequences, Discrepancies and Applications},
    SERIES = {Lecture Notes in Mathematics},
    VOLUME = {1651},
 PUBLISHER = {Springer-Verlag, Berlin},
      YEAR = {1997},
     PAGES = {xiv+503},
      ISBN = {3-540-62606-9},
   MRCLASS = {11Kxx (11K06 11K38)},
  MRNUMBER = {1470456},
MRREVIEWER = {Oto\ Strauch},
       DOI = {10.1007/BFb0093404},
       URL = {https://doi.org/10.1007/BFb0093404},
}

@article {H21,
    AUTHOR = {He, X.},
     TITLE = {Lattice-based designs with quasi-optimal separation distance on all projections},
   JOURNAL = {Biometrika},
  FJOURNAL = {Biometrika},
    VOLUME = {108},
      YEAR = {2021},
    NUMBER = {2},
     PAGES = {443--454},
      ISSN = {0006-3444,1464-3510},
   MRCLASS = {62K05 (05B30 62K15 68W40)},
  MRNUMBER = {4259142},
       DOI = {10.1093/biomet/asaa057},
       URL = {https://doi.org/10.1093/biomet/asaa057},
}

@book {KN74,
    AUTHOR = {Kuipers, L. and Niederreiter, H.},
     TITLE = {Uniform Distribution of Sequences},
    SERIES = {Pure and Applied Mathematics},
 PUBLISHER = {Wiley-Interscience [John Wiley \& Sons], New
              York-London-Sydney},
      YEAR = {1974},
     PAGES = {xiv+390},
   MRCLASS = {10K05 (22D99)},
  MRNUMBER = {419394},
MRREVIEWER = {P.\ Gerl},
}

@article {lar86,
    AUTHOR = {Larcher, G.},
     TITLE = {{\"U}ber die isotrope {D}iskrepanz von {F}olgen},
   JOURNAL = {Arch. Math. (Basel)},
  FJOURNAL = {Archiv der Mathematik},
    VOLUME = {46},
      YEAR = {1986},
    NUMBER = {3},
     PAGES = {240--249},
      ISSN = {0003-889X,1420-8938},
   MRCLASS = {11K38},
  MRNUMBER = {834843},
MRREVIEWER = {L.\ Kuipers},
       DOI = {10.1007/BF01194190},
       URL = {https://doi.org/10.1007/BF01194190},
}

@article {L88,
    AUTHOR = {Larcher, G.},
     TITLE = {On the distribution of {$s$}-dimensional {K}ronecker-sequences},
   JOURNAL = {Acta Arith.},
  FJOURNAL = {Polska Akademia Nauk. Instytut Matematyczny. Acta Arithmetica},
    VOLUME = {51},
      YEAR = {1988},
    NUMBER = {4},
     PAGES = {335--347},
      ISSN = {0065-1036},
   MRCLASS = {11K38 (11J71)},
  MRNUMBER = {971085},
MRREVIEWER = {G.\ Turnwald},
       DOI = {10.4064/aa-51-4-335-347},
       URL = {https://doi.org/10.4064/aa-51-4-335-347},
}

@book {N92,
    AUTHOR = {Niederreiter, H.},
     TITLE = {Random number generation and quasi-{M}onte {C}arlo methods},
    SERIES = {CBMS-NSF Regional Conference Series in Applied Mathematics},
    VOLUME = {63},
 PUBLISHER = {Society for Industrial and Applied Mathematics (SIAM),
              Philadelphia, PA},
      YEAR = {1992},
}

@incollection {N73,
    AUTHOR = {Niederreiter, H.},
     TITLE = {Application of {D}iophantine approximations to numerical
              integration},
 BOOKTITLE = {Diophantine approximation and its applications ({P}roc.
              {C}onf., {W}ashington, {D}.{C}., 1972)},
     PAGES = {129--199},
 PUBLISHER = {Academic Press, New York-London},
      YEAR = {1973},
   MRCLASS = {10K05 (65D30)},
  MRNUMBER = {357357},
MRREVIEWER = {O.\ P.\ Stackelberg},
}

@ARTICLE{roth1,
  author = {Roth, K.~F.},
  year = 1954,
  title = {On irregularities of distribution},
  journal = {Mathematika},
  volume = {1},
  pages = {73--79}
}

@article {SP21,
    AUTHOR = {Sonnleitner, M. and Pillichshammer, F.},
     TITLE = {On the relation of the spectral test to isotropic discrepancy
              and {$L_q$}-approximation in {S}obolev spaces},
   JOURNAL = {J. Complexity},
  FJOURNAL = {Journal of Complexity},
    VOLUME = {67},
      YEAR = {2021},
     PAGES = {Paper No. 101576, 9},
      ISSN = {0885-064X,1090-2708},
   MRCLASS = {11K38 (46E35 52A39)},
  MRNUMBER = {4311529},
MRREVIEWER = {\'Isabel\ Pirsic},
       DOI = {10.1016/j.jco.2021.101576},
       URL = {https://doi.org/10.1016/j.jco.2021.101576},
}

@article {PS20,
    AUTHOR = {Pillichshammer, F. and Sonnleitner, M.},
     TITLE = {A note on isotropic discrepancy and spectral test of lattice
              point sets},
   JOURNAL = {J. Complexity},
  FJOURNAL = {Journal of Complexity},
    VOLUME = {58},
      YEAR = {2020},
     PAGES = {101441, 7},
      ISSN = {0885-064X,1090-2708},
   MRCLASS = {11K38},
  MRNUMBER = {4077829},
MRREVIEWER = {\'Isabel\ Pirsic},
       DOI = {10.1016/j.jco.2019.101441},
       URL = {https://doi.org/10.1016/j.jco.2019.101441},
}

@book {Siegel89,
    AUTHOR = {Siegel, C. L.},
     TITLE = {Lectures on the Geometry of Numbers},
 PUBLISHER = {Springer Science \& Business Media},
      YEAR = {1989},
}

@incollection {GK09,
    AUTHOR = {Gr\"unschlo{\ss}, L. and Keller, A.},
     TITLE = {{$(t,m,s)$}-nets and maximized minimum distance, {P}art {II}},
 BOOKTITLE = {Monte {C}arlo and quasi-{M}onte {C}arlo methods 2008},
     PAGES = {395--409},
 PUBLISHER = {Springer, Berlin},
      YEAR = {2009},
}

@incollection {GHSK08,
    AUTHOR = {Gr\"unschlo{\ss}, L. and Hanika, J. and Schwede, R. and Keller, A.},
     TITLE = {{$(t,m,s)$}-nets and maximized minimum distance},
 BOOKTITLE = {Monte {C}arlo and quasi-{M}onte {C}arlo methods 2006},
     PAGES = {397--412},
 PUBLISHER = {Springer, Berlin},
      YEAR = {2008},
}

@article {NS94,
    AUTHOR = {Niederreiter, H. and Sloan, I.~H.},
     TITLE = {Integration of nonperiodic functions of two variables by
              {F}ibonacci lattice rules},
   JOURNAL = {J. Comput. Appl. Math.},
  FJOURNAL = {Journal of Computational and Applied Mathematics},
    VOLUME = {51},
      YEAR = {1994},
    NUMBER = {1},
     PAGES = {57--70},
}

@article {SS07,
    AUTHOR = {Sobol, I. M. and Shukhman, B. V.},
     TITLE = {Quasi-random points keep their distance},
   JOURNAL = {Math. Comput. Simulation},
  FJOURNAL = {Mathematics and Computers in Simulation},
    VOLUME = {75},
      YEAR = {2007},
    NUMBER = {3-4},
     PAGES = {80--86},
}

@article {G24a,
    AUTHOR = {Goda, T.},
     TITLE = {The {S}obol' sequence is not quasi-uniform in dimension 2},
   JOURNAL = {Proc. Amer. Math. Soc.},
  FJOURNAL = {Proceedings of the American Mathematical Society},
    VOLUME = {152},
      YEAR = {2024},
    NUMBER = {8},
     PAGES = {3209--3213},
}

@article {G24b,
    AUTHOR = {Goda, T.},
     TITLE = {One-dimensional quasi-uniform {K}ronecker sequences},
   JOURNAL = {Arch. Math.},
  FJOURNAL = {Archiv der Mathematik},
    VOLUME = {123},
      YEAR = {2024},
    NUMBER = {5},
     PAGES = {499--505},
}

@book {LePi14,
    AUTHOR = {Leobacher, G. and Pillichshammer, F.},
     TITLE = {Introduction to Quasi-{M}onte {C}arlo Integration and
              Applications},
    SERIES = {Compact Textbooks in Mathematics},
 PUBLISHER = {Birkh\"auser/Springer, Cham},
      YEAR = {2014},
     PAGES = {xii+195},
      ISBN = {978-3-319-03424-9; 978-3-319-03425-6},
   MRCLASS = {11K06 (11K38 11K45 60H99 65C05 65D30)},
  MRNUMBER = {3289176},
MRREVIEWER = {Robert\ F.\ Tichy},
       DOI = {10.1007/978-3-319-03425-6},
       URL = {https://doi.org/10.1007/978-3-319-03425-6},
}

@article {PZ23,
    AUTHOR = {Pronzato, L. and Zhigljavsky, A.},
     TITLE = {Quasi-uniform designs with optimal and near-optimal uniformity
              constant},
   JOURNAL = {J. Approx. Theory},
  FJOURNAL = {Journal of Approximation Theory},
    VOLUME = {294},
      YEAR = {2023},
     PAGES = {Paper No. 105931, 14},
}

@book {SJ94,
    AUTHOR = {Sloan, I.~H. and Joe, S.},
     TITLE = {Lattice Methods for Multiple Integration},
    PUBLISHER = {Oxford University Press, New York},
      YEAR = {1994},
     PAGES = {xii+239},
}

@article {SW06,
    AUTHOR = {Schaback, R. and Wendland, H.},
     TITLE = {Kernel techniques: {F}rom machine learning to meshless methods},
   JOURNAL = {Acta Numer.},
  FJOURNAL = {Acta Numerica},
    VOLUME = {15},
      YEAR = {2006},
     PAGES = {543--639},
}

@book {W05,
    AUTHOR = {Wendland, H.},
     TITLE = {Scattered Data Approximation},
    SERIES = {Cambridge Monographs on Applied and Computational Mathematics},
    VOLUME = {17},
 PUBLISHER = {Cambridge University Press, Cambridge},
      YEAR = {2005},
}

@article {T20,
    AUTHOR = {Teckentrup, A.~L.},
     TITLE = {Convergence of {G}aussian process regression with estimated
              hyper-parameters and applications in {B}ayesian inverse
              problems},
   JOURNAL = {SIAM/ASA J. Uncertain. Quantif.},
  FJOURNAL = {SIAM/ASA Journal on Uncertainty Quantification},
    VOLUME = {8},
      YEAR = {2020},
    NUMBER = {4},
     PAGES = {1310--1337},
}

@article {WBG21,
    AUTHOR = {Wynne, G. and Briol, F.-X. and Girolami,
              M.},
     TITLE = {Convergence guarantees for {G}aussian process means with
              misspecified likelihoods and smoothness},
   JOURNAL = {J. Mach. Learn. Res.},
  FJOURNAL = {Journal of Machine Learning Research (JMLR)},
    VOLUME = {22},
      YEAR = {2021},
     PAGES = {Paper No. 123, 40},
}

@book {SWN03,
    AUTHOR = {Santner, T.~J. and Williams, B.~J. and Notz, W.~I.},
     TITLE = {The design and analysis of computer experiments},
    SERIES = {Springer Series in Statistics},
 PUBLISHER = {Springer-Verlag, New York},
      YEAR = {2003},
     PAGES = {xii+283},
}

@book {FLS06,
    AUTHOR = {Fang, K.-T. and Li, R. and Sudjianto, A.},
     TITLE = {Design and Modeling for Computer Experiments},
    SERIES = {Chapman \& Hall/CRC Computer Science and Data Analysis Series},
 PUBLISHER = {Chapman \& Hall/CRC, Boca Raton, FL},
      YEAR = {2006},
     PAGES = {xii+290},
}

@article {PM12,
    AUTHOR = {Pronzato, L. and M\"uller, W.~G.},
     TITLE = {Design of computer experiments: space filling and beyond},
   JOURNAL = {Stat. Comput.},
  FJOURNAL = {Statistics and Computing},
    VOLUME = {22},
      YEAR = {2012},
    NUMBER = {3},
     PAGES = {681--701},
}

@article {F80,
    AUTHOR = {Frolov, K.~K.},
     TITLE = {Upper bound of the discrepancy in metric {$L\sb{p}$}, {$2\leq
              p<\infty $}},
   JOURNAL = {Dokl. Akad. Nauk SSSR},
  FJOURNAL = {Doklady Akademii Nauk SSSR},
    VOLUME = {252},
      YEAR = {1980},
    NUMBER = {4},
     PAGES = {805--807},
}

@article {F76,
    AUTHOR = {Frolov, K.~K.},
     TITLE = {Upper bounds for the errors of quadrature formulae on classes
              of functions},
   JOURNAL = {Dokl. Akad. Nauk SSSR},
  FJOURNAL = {Doklady Akademii Nauk SSSR},
    VOLUME = {231},
      YEAR = {1976},
    NUMBER = {4},
     PAGES = {818--821},
}

@incollection {U16,
    AUTHOR = {Ullrich, M.},
     TITLE = {On ``{U}pper error bounds for quadrature formulas on function
              classes'' by {K}. {K}. {F}rolov},
 BOOKTITLE = {Monte {C}arlo and quasi-{M}onte {C}arlo methods},
    SERIES = {Springer Proc. Math. Stat.},
    VOLUME = {163},
     PAGES = {571--582},
 PUBLISHER = {Springer, Cham},
      YEAR = {2016},
}

@article {S94,
    AUTHOR = {Skriganov, M.~M.},
     TITLE = {Constructions of uniform distributions in terms of geometry of
              numbers},
   JOURNAL = {Algebra i Analiz},
  FJOURNAL = {Rossi\u iskaya Akademiya Nauk. Algebra i Analiz},
    VOLUME = {6},
      YEAR = {1994},
    NUMBER = {3},
     PAGES = {200--230},
}

@article{eisenbrand2001short,
    AUTHOR = {Eisenbrand, F.},
     TITLE = {Short vectors of planar lattices via continued fractions},
   JOURNAL = {Inform. Process. Lett.},
  FJOURNAL = {Information Processing Letters},
    VOLUME = {79},
      YEAR = {2001},
    NUMBER = {3},
     PAGES = {121--126},
  publisher={Elsevier}
}

@article {Weyl16,
    AUTHOR = {Weyl, H.},
     TITLE = {{\"U}ber die {G}leichverteilung von {Z}ahlen mod. {E}ins},
   JOURNAL = {Math. Ann.},
  FJOURNAL = {Mathematische Annalen},
    VOLUME = {77},
      YEAR = {1916},
    NUMBER = {3},
     PAGES = {313--352},
      ISSN = {0025-5831,1432-1807},
   MRCLASS = {99-04},
  MRNUMBER = {1511862},
       DOI = {10.1007/BF01475864},
       URL = {https://doi.org/10.1007/BF01475864},
}

@article {ZKH09,
    AUTHOR = {Zeng, X. and Kritzer, P. and Hickernell, F.~J.},
     TITLE = {Spline methods using integration lattices and digital nets},
   JOURNAL = {Constr. Approx.},
  FJOURNAL = {Constructive Approximation. An International Journal for
              Approximations and Expansions},
    VOLUME = {30},
      YEAR = {2009},
    NUMBER = {3},
     PAGES = {529--555},
      ISSN = {0176-4276,1432-0940},
   MRCLASS = {41A15 (41A63 42B05 65D07)},
  MRNUMBER = {2558692},
       DOI = {10.1007/s00365-009-9072-0},
       URL = {https://doi.org/10.1007/s00365-009-9072-0},
}

@misc {DGSxx,
    author = {Dick, J. and Goda, T. and Suzuki, K.},
    title = {On the quasi-uniformity properties of quasi-{M}onte {C}arlo point sets and sequences -- {P}art~{II}: {D}igital nets and sequences},
    note={arXiv preprint arXiv:2501.18226v2},
    year={2025}
}

@article {S70,
    AUTHOR = {Schmidt, W. M.},
     TITLE = {Simultaneous approximation to algebraic numbers by rationals},
   JOURNAL = {Acta Math.},
  FJOURNAL = {Acta Mathematica},
    VOLUME = {125},
      YEAR = {1970},
     PAGES = {189--201},
      ISSN = {0001-5962,1871-2509},
   MRCLASS = {10.30},
  MRNUMBER = {268129},
MRREVIEWER = {K.\ Mahler},
       DOI = {10.1007/BF02392334},
       URL = {https://doi.org/10.1007/BF02392334},
}

@ARTICLE{schm72,
  author = {Schmidt, W.~M.},
  year = {1972},
  title = {Irregularities of distribution {VII}},
  journal = {Acta Arith.},
  volume = {21},
  pages = {45--50}
}

@ARTICLE{schm75,
  author = {Schmidt, W.~M.},
  year = {1975},
  title = {Irregularities of distribution. {IX}},
  journal = {Acta Arith.},
  volume = {27},
  pages = {385--396}
}

@incollection {schm77,
    AUTHOR = {Schmidt, W.~M.},
     TITLE = {Irregularities of distribution. {X}},
 BOOKTITLE = {Number Theory and Algebra},
     PAGES = {311--329},
 PUBLISHER = {Academic Press},
   ADDRESS = {New York},
      YEAR = {1977}
}

@ARTICLE{stute,
  author = {Stute, W.},
  year = 1977,
  title = {Convergence rates for the isotrope discrepancy},
  journal = {Ann. Probability},
  volume = {5},
  number ={5},
  pages = {707--723}
}

@book {Khi,
    AUTHOR = {Khinchin, A. Ya.},
     TITLE = {Continued Fractions},
   EDITION = {Russian},
      NOTE = {With a preface by B. V. Gnedenko,
              Reprint of the 1964 translation},
 PUBLISHER = {Dover Publications, Inc., Mineola, NY},
      YEAR = {1997},
     PAGES = {xii+95},
      ISBN = {0-486-69630-8},
   MRCLASS = {11A55 (01A75 11-03 11J70)},
  MRNUMBER = {1451873},
}

@article {pausinger2017bounds,
    AUTHOR = {Pausinger, F.},
     TITLE = {Bounds for the traveling salesman paths of two-dimensional
              modular lattices},
   JOURNAL = {J. Comb. Optim.},
  FJOURNAL = {Journal of Combinatorial Optimization},
    VOLUME = {33},
      YEAR = {2017},
    NUMBER = {4},
     PAGES = {1378--1394},
}

@article {pausinger2017lattices,
    AUTHOR = {Pausinger, F.},
     TITLE = {Lattices modulo {$N$} with long shortest distances},
   JOURNAL = {Elem. Math.},
  FJOURNAL = {Elemente der Mathematik},
    VOLUME = {72},
      YEAR = {2017},
    NUMBER = {3},
     PAGES = {111--121},
}

@article {SY19,
    AUTHOR = {Suzuki, K. and Yoshiki, T.},
     TITLE = {Enumeration of the {C}hebyshev-{F}rolov lattice points in
              axis-parallel boxes},
   JOURNAL = {Hiroshima Math. J.},
  FJOURNAL = {Hiroshima Mathematical Journal},
    VOLUME = {49},
      YEAR = {2019},
    NUMBER = {1},
     PAGES = {139--159},
      ISSN = {0018-2079,2758-9641},
   MRCLASS = {11P21 (11Y16 65D30)},
  MRNUMBER = {3936651},
MRREVIEWER = {\'Isabel\ Pirsic},
       DOI = {10.32917/hmj/1554516041},
       URL = {https://doi.org/10.32917/hmj/1554516041},
}

@article {Sch95,
    AUTHOR = {Schaback, R.},
     TITLE = {Error estimates and condition numbers for radial basis function interpolation},
   JOURNAL = {Adv. Comput. Math.},
  FJOURNAL = {Advances in Computational Mathematics},
    VOLUME = {3},
      YEAR = {1995},
    NUMBER = {3},
     PAGES = {251--264},
}

@article {KOUU21,
    AUTHOR = {Kacwin, C. and Oettershagen, J. and Ullrich, M.
              and Ullrich, T.},
     TITLE = {Numerical performance of optimized {F}rolov lattices in tensor
              product reproducing kernel {S}obolev spaces},
   JOURNAL = {Found. Comput. Math.},
  FJOURNAL = {Foundations of Computational Mathematics. The Journal of the
              Society for the Foundations of Computational Mathematics},
    VOLUME = {21},
      YEAR = {2021},
    NUMBER = {3},
     PAGES = {849--889},
 }

@article {C24,
    AUTHOR = {Chkifa, M.~A.},
     TITLE = {Lattice enumeration via linear programming},
   JOURNAL = {Numer. Math.},
  FJOURNAL = {Numerische Mathematik},
    VOLUME = {156},
      YEAR = {2024},
    NUMBER = {1},
     PAGES = {71--106},
}

@article {WSH21,
    AUTHOR = {Wenzel, T. and Santin, G. and Haasdonk, B.},
     TITLE = {A novel class of stabilized greedy kernel approximation algorithms: {C}onvergence, stability and uniform point distribution},
   JOURNAL = {J. Approx. Theory},
  FJOURNAL = {Journal of Approximation Theory},
    VOLUME = {262}, 
      YEAR = {2021},
     PAGES = {Paper No. 105508, 30},
       DOI = {10.1016/j.jat.2020.105508},
}

@article {JMY90,
    AUTHOR = {Johnson, M. E. and Moore, L. M. and Ylvisaker, D.},
     TITLE = {Minimax and maximin distance designs},
   JOURNAL = {J. Statist. Plann. Inference},
  FJOURNAL = {Journal of Statistical Planning and Inference},
    VOLUME = {26},
      YEAR = {1990},
    NUMBER = {2},
     PAGES = {131--148},
       DOI = {10.1016/0378-3758(90)90122-B},
}

@article {P17,
    AUTHOR = {Pronzato, L.},
     TITLE = {Minimax and maximin space-filling designs: some properties and
              methods for construction},
   JOURNAL = {J. SFdS},
  FJOURNAL = {Journal de la SFdS. Journal de la Soc\'iet\'e{} Fran\c caise
              de Statistique},
    VOLUME = {158},
      YEAR = {2017},
    NUMBER = {1},
     PAGES = {7--36},
}

@article {HRDH11,
    AUTHOR = {Husslage, B. G. M. and Rennen, G. and van Dam, E. R. and Den Hertog, D.},
     TITLE = {Space-filling {L}atin hypercube designs for computer experiments},
   JOURNAL = {Optim. Eng.},
  FJOURNAL = {Optimization and Engineering },
    VOLUME = {12},
      YEAR = {2011},
    NUMBER = {1},
     PAGES = {610--633},
}

@article {J16,
    AUTHOR = {Joseph, V.~R.},
     TITLE = {{Space-filling designs for computer experiments: A review}},
   JOURNAL = {Qual. Eng.},
  FJOURNAL = {Quality Engineering},
    VOLUME = {28},
      YEAR = {2016},
    NUMBER = {1},
     PAGES = {28--35},
       DOI = {10.1080/08982112.2015.1100447},
}

@article {DS10,
    AUTHOR = {De Marchi, S. and Schaback, R.},
     TITLE = {Stability of kernel-based interpolation},
   JOURNAL = {Adv. Comput. Math.},
  FJOURNAL = {Advances in Computational Mathematics},
    VOLUME = {32},
      YEAR = {2010},
    NUMBER = {2},
     PAGES = {155--161},
       DOI = {10.1007/s10444-008-9093-4},
}

@article {NWW05,
    AUTHOR = {Narcowich, F.~J. and Ward, J.~D. and Wendland,
              H.},
     TITLE = {Sobolev bounds on functions with scattered zeros, with
              applications to radial basis function surface fitting},
   JOURNAL = {Math. Comp.},
  FJOURNAL = {Mathematics of Computation},
    VOLUME = {74},
      YEAR = {2005},
    NUMBER = {250},
     PAGES = {743--763},
       DOI = {10.1090/S0025-5718-04-01708-9},
}

@misc{suzuki2025a,
      title={Exact $\ell^\infty$-separation radius of Sobol' sequences in dimension 2}, 
      author={Kosuke Suzuki},
      year={2025},
      eprint={2508.14803},
      archivePrefix={arXiv},
      primaryClass={math.NA},
      url={https://arxiv.org/abs/2508.14803}, 
      note={arXiv preprint arXiv:2508.14803}
}

\end{document}